%% file: main-percolation-mdps.tex
\documentclass{article}

\input{preamble}

\title{Percolation Markov Decision Processes}

\author{%
    Melissa González García\\[0.3em]
    \small\textit{CEREMADE, Université Paris Dauphine-PSL} 
    \and
    Guillaume Vigeral\\[0.3em]
    \small\textit{Centre d’Economie de la Sorbonne, CNRS, Université Paris 1 Panthéon-Sorbonne}
}

\date{}

\begin{document}
    
	\maketitle
    
    \begin{abstract}
        We study Percolation Markov Decision Processes (PMDPs), where the decision maker repeatedly moves a token through the $d$-dimensional integer lattice with deterministic transitions and random payoffs assigned to the edges. Payoffs are revealed before the beginning of the decision process and the decision maker aims to maximize the average accumulated payoff over a fixed horizon. First, we address general asymptotic control-theoretic questions by proving the existence of the uniform value and 0-optimal strategies. Furthermore, we examine the structural properties of optimal strategies: we establish the existence of stationary 0-optimal strategies for a class where the token moves in a fixed direction independent of the player's actions, and we provide low-dimensional examples to illustrate their complexity. Moreover, we introduce a dimensional lifting property that allows PMDPs to be approximated by lower-dimensional PMDPs, offering a representation of the approximation of critical probability thresholds in percolation theory via PMDPs.
	\end{abstract}
    
	\tableofcontents

\section{Introduction}\label{sec:intro} 
    \emph{Percolation games}, introduced by \citet{GarnierZiliotto2022}, are a variation of \emph{zero-sum stochastic games} (discrete-time dynamic games introduced by \citet{Shapley1953}) in a stationary and ergodic random environment. In these games, two players take turns to move a token along the vertices of $\mathbb{Z}^d$ over discrete stages, and random payoffs are assigned to the edges. Before the game begins, the realization of the payoffs is revealed to both players. Given an initial state $\BFz \in \mathbb{Z}^d$ and a horizon $n \in \mathbb{N}$, the corresponding $n$-stage game has player~1 seeking to maximize the average payoff accumulated over the first $n$ stages, while player~2 seeks to minimize it; it is a \emph{zero-sum game}. Let $v_n(\BFz)$ denote the (outcome) \emph{value} of this game. A foundational result of \citet[Theorem 2.1]{GarnierZiliotto2022} establishes that, when the dynamics are \emph{oriented} (that is, every token's move has a positive projection onto a direction vector) and the payoffs are independent and identically distributed (i.i.d.) across space, $v_n(\BFz)$ converges almost surely as $n \to \infty$ to a deterministic \emph{limit value} that is independent of the initial position of the token. Removing the orientation assumption introduces substantial difficulties, as shown by \citet{sepulveda2024gameorientedpercolation}, who studied a simple yet remarkably rich non-oriented percolation game for which the existence of a limit value remains open.  This model is interesting for several reasons. Despite the extensive literature on long-horizon zero-sum stochastic games (general references are, e.g., \citep{Shapley1953, BewleyKohlberg1976, Chatterjee2009ASO, SolanVieille2015}), positive results on the existence of limit values for games with countable state spaces remain scarce. This makes percolation games a particularly appealing model. They also belong to the broader class of \emph{random games}, in which a game is randomly selected, revealed to the players, and then played (see, e.g., \citep{Amiet2021, Flesch2023, Attia2025} for recent references). 
    
    In this paper, we study the single-player version of percolation games with i.i.d. payoffs on the edges, which we refer to as percolation Markov decision processes (PMDPs). Once a realization of the payoffs is fixed, a PMDP can be interpreted as an MDP \citep{blackwell62, HernndezLerma1996, Puterman2005-ah} with state space $\mathbb{Z}^d$ and state-independent, deterministic transitions. The existence of the limit value, a uniform value (a limit value that can be approximately guaranteed for all sufficiently long finite horizons) and $0$-optimal strategies (strategies that guarantee this value in the limit without any error) for MDPs with countable state spaces, and the subsequent identification of simple 0-optimal strategies, are complex questions and important issues for applications (see, e.g., \citep{Renault2011, kiefer2020how}, \citep[Section 2.5.3]{LarakiSorin2015}, and the references therein). In this paper, we address these and other questions for PMDPs. Our main results are summarized as follows.
    
    In PMDPs, aside from the oriented case, we first show that there are only two other possibilities: either the PMDP is trivial (meaning that the token never moves), or it is controllable (in the sense that there exists a finite sequence of actions that allows the decision maker to cycle indefinitely through finitely many positions). The existence of the limit value for oriented PMDPs follows from the corresponding result for oriented percolation games \citep[Theorem 2.1]{GarnierZiliotto2022}. As a baseline result, Theorem~\ref{thm:limit-value} establishes the existence of the limit value for PMDPs in the remaining two regimes. We then build upon this to prove the existence of a \emph{uniform value}, meaning that there is a strategy that nearly guarantees the limit value for all sufficiently large horizons. Most notably—and perhaps surprisingly given the non-compactness of the state space—our main result, Theorem~\ref{thm:main-uniform}, establishes the existence of a \emph{$0$-optimal strategy} guaranteeing asymptotic attainment of this value. More precisely, we construct a block-by-block strategy that ensures the limit value in the infinite-horizon PMDP, in which the payoff is defined as the limit inferior of the average of the stage payoffs. In particular, we show that the values of the infinite-horizon PMDP with limit inferior and limit superior criteria coincide with the limit value and thus that the uniform value exists. Notably, as noted in Remark~\ref{rem:extension-2-players}, the core ideas of the block-by-block construction of the strategy naturally extend to the two-player oriented framework.
    
    How simple can a $0$-optimal strategy be?~In Theorem~\ref{thm:oriented-stationary} we demonstrate that 0-optimal strategies for oriented PMDPs can be chosen to be \emph{stationary}, that is, depending solely on the current position of the token at each stage.~This is not the case for controllable PMDPs, as explained in Remark~\ref{rem:no0stationary}. Moreover, we explicitly construct stationary 0-optimal strategies and compute the uniform value for certain one-dimensional oriented PMDPs with Bernoulli payoffs, showing that even simple settings can exhibit complex optimal behavior.
    
    In Theorem~\ref{thm:convergence_game_m} we introduce the \emph{dimensional lifting} property, establishing that high-dimensional PMDPs can be approximated via lower-dimensional ones. To illustrate its potential use, we characterize the critical probability threshold of oriented Bernoulli percolation on the square lattice \citep{Durrett1984} via the uniform value of a two-dimensional PMDP. This characterization is close in spirit to the one in \citep[Theorem 5.1]{sepulveda2024gameorientedpercolation} where the authors characterize the same parameter by a non-oriented percolation game. This is simplified here via a continuity result. One main novelty lies in introducing the dimensional lifting property and using it to represent an approximation of a critical parameter via a sequence of one-dimensional PMDPs.

    Finally, oriented PMDPs naturally link with \emph{directed last-passage percolation} (LPP) (another classical model in probability \citep[see][]{DamronRassoulASepp2016}). For interested readers, we briefly review these probabilistic models and explore their connection to PMDPs.

    \subsubsection*{Outline of the paper} In Section \ref{sec:model}, we introduce the model in detail, classify PMDPs into the three possible regimes, and state our main results. Next, in Section \ref{sec:proof_main_results}, we prove our primary asymptotic control-theoretic results: the existence of the limit value (Theorem \ref{thm:limit-value}), the existence of a $0$-optimal strategy (Theorem \ref{thm:main-uniform}), and the existence of $0$-optimal stationary strategies for oriented PMDPs (Theorem \ref{thm:oriented-stationary}). In Section \ref{sec:dimensional-lifting-oriented-percolation}, we prove our final main result: the dimensional lifting property. In Theorem~\ref{thm:convergence_game_m}, we establish this property for a sequence of one-dimensional, two-action PMDPs and outline possible extensions. We then link a limiting two-dimensional PMDP and oriented Bernoulli percolation alongside a continuity result for PMDPs with Bernoulli payoffs. Then, we present simulations that leverage the dimensional lifting property to approximate the critical parameter of the percolation model, and of three related examples. Then, in Section \ref{sec:dimension1}, we explicitly determine the uniform value for specific one-dimensional oriented PMDPs with Bernoulli payoffs (relegating the proof of the most complex case to the appendix) and introduce a natural preorder on positive action sets from which we derive several relations. In Section \ref{sec:lpp}, we briefly present LPP models and discuss their connections to PMDPs. Finally, in Section \ref{sec:perspectives}, we outline future research directions.

    \subsubsection*{Notation} 
    $\bbN$ denotes the set of non-negative integers and $\bbN_+ := \bbN \setminus \{\BFzero\}$. For $\BFx \in \bbR^d$, $|\BFx| = (x_1^2 + \cdots + x_d^2)^{1/2}$ denotes the Euclidean norm. The origin is denoted by $\BFzero$, the standard basis vectors of $\bbR^d$ are $\BFe_1 = (1, 0, \ldots, 0), \BFe_2 = (0, 1, \ldots, 0), \ldots, \BFe_d = (0, 0, \ldots, 1)$, and $\BFone_d = (1, \ldots, 1)$. The Borel $\sigma$-algebra is denoted by $\mathcal{B}(\bbR)$ and $\mathcal{P}(\bbZ^d)$ denotes the power set of $\bbZ^d$. Let $(\bbZ^d,\E^d)$ be the graph with vertex set $\bbZ^d$, where two vertices $\BFx,\BFy \in \bbZ^d$ are connected by an edge if and only if $|\BFx - \BFy| = 1$; when $d = 2$, this graph is called the \emph{square lattice}. Let $(\Omega, \mathcal{F}, \mu)$ be a measure space and let $f: \Omega \to \bbR$ be a measurable function. The \emph{essential supremum} of $f$ is defined by $\operatorname*{ess\,sup}(f) := \inf \{a \in \bbR : \mu(f^{-1}(a, \infty)) = 0\}$. Analogously, $\operatorname*{ess\,inf}(f) := \sup \{a \in \bbR : \mu(f^{-1}(-\infty, a)) = 0\}$. For $a,b \in \bbR$, we write $a \gg b$ to indicate that $a$ is much larger than $b$.
	
\section{Model and main results}\label{sec:model}
    \subsection{The model}
    We focus on \emph{percolation MDPs} (PMDPs) derived from the \emph{percolation games} model introduced by \citet{GarnierZiliotto2022}, by restricting one player's set of actions to a singleton. For completeness and clarity, we present it in detail, following a structure similar to that of the original model.

    A PMDP on $\bbZ^d$ is described by a tuple $(I, \mathcal{E}, g)$, where
    \begin{itemize}[noitemsep, topsep=0pt]
        \item $I \subset \bbZ^d$ is the non-empty, state-independent and finite set of actions of the decision maker.
        \item $\mathcal{E} = (\Omega, \F, \Pbb)$ is a probability space. 
        \item $g := \{\omega \mapsto g_{\omega}(\BFz,\BFa)\}_{(\BFz,\BFa) \in \bbZ^d \times I}$ is a collection of uniformly bounded, independent and identically distributed (i.i.d.) $\bbR$-valued random payoffs defined on $\Omega$, with expectation $\E g$.
    \end{itemize}

    Given some initial state $\BFz \in \bbZ^d$ and $\omega \in \Omega$, the decision process proceeds as follows:
    \begin{itemize}[noitemsep,topsep=0pt]
        \item~The decision maker is informed of the realization of the payoffs given by $\omega$.
        \item~At every stage $m \ge 1$, the decision maker observes the current state $\BFz_m$ (where $\BFz_1 = \BFz$), chooses an action $\BFa_m$, and receives the payoff $g_{\omega}(\BFz_m, \BFa_m)$. Then, $\BFz_{m + 1}$ is determined by $\BFz_{m+1} = \BFz_m + \BFa_m$.
    \end{itemize}
    
    For $\BFz \in \bbZ^d$ the PMDP starting at $\BFz$ is denoted by $\Gamma(\BFz)$.
        
    \begin{remark}\label{rem:MDPs}
        Once $\omega$ is fixed, there is no randomness. The decision maker faces a deterministic optimization problem with full knowledge of the payoffs. In this sense, it is a Markov Decision Process (MDP) with state-independent and deterministic transitions \citep{HernndezLerma1996, Puterman2005-ah} or a dynamic programming problem \citep[Section 1]{Renault2011}.
    \end{remark}
    
    We define the \emph{observable environment} $\mathcal{F}_{\mathrm{all}} \subseteq \F$ as the $\sigma$-algebra generated by the payoff field; that is, $\mathcal{F}_{\mathrm{all}} = \sigma\big( \{ \omega \mapsto g_\omega(\BFz, \BFa) : \BFz \in \bbZ^d, \BFa \in I\} \big)$. A \emph{pure environment-dependent strategy} is a sequence of mappings $\sigma = (\sigma_m)_{m\geq 1}$, where for each $m \geq 1$:
    \begin{itemize}[noitemsep,topsep=0pt]
        \item[\textup{(i)}] $\sigma_m(\omega, \BFz) \in I$ for all $(\omega, \BFz) \in \Omega \times \bbZ^d$.
        \item[\textup{(ii)}] For each $m \geq 1$, $\sigma_m$ is measurable with respect to the product $\sigma$-algebra $\mathcal{F}_{\mathrm{all}} \otimes \mathcal{P}(\bbZ^d)$.
    \end{itemize}
    The set of such strategies is denoted by $\Sigma$. Every strategy $\sigma \in \Sigma$ and initial state $\BFz \in \bbZ^d$ induce a play: $(\BFz_m, \BFa_m)_{m \geq 1}$ with $\BFz_1 := \BFz$, $\BFa_1 := \sigma_1(\omega,\BFz)$, and $\BFz_m=\BFz_{m-1}+\BFa_{m-1}$, $\BFa_m=\sigma_m(\omega,\BFz), m \geq 2$. Intuitively, the strategy specifies the action at every stage based on the initial position of the token and the full knowledge of the payoff configuration given by $\omega$. 

    The $n$-stage decision process, denoted by $\Gamma_n^{\omega}(\BFz)$, has payoff function $\gamma_n^{\omega} \colon \bbZ^d \times \Sigma \to \bbR$ defined by
    \[
        \gamma_n^{\omega}(\BFz, \sigma) := \frac{1}{n}\sum_{m=1}^ng_{\omega}(\BFz_m, \BFa_m),
    \]
    and value function
    \[
        v_n^{\omega}(\BFz) := \max_{\sigma \in \Sigma} \gamma_n^{\omega}(\BFz, \sigma).
    \]
    In the sequel, $v_n(\BFz) \colon \mathcal{E} \to (\bbR, \mathcal{B}(\bbR))$ denotes the random variable $\omega \mapsto v_n^{\omega}(\BFz)$.

    It is natural to consider the corresponding model with payoffs on the vertices.

    \begin{definition}[Vertex-based payoffs model]\label{def:def_onthevertices}
        A \emph{PMDP  with payoffs on the vertices} is defined as a variation of the edge-payoffs model by letting $g(\BFz, \BFa) = g(\BFz + \BFa)$ for all $\BFz \in \bbZ^d$ and $\BFa \in I$, with $\{\omega \mapsto g_{\omega}(\BFz)\}_{\BFz \in \bbZ^d}$ a collection of uniformly bounded i.i.d. $\bbR$-valued random variables.
    \end{definition}

    This alternative model is not a particular case of PMDPs since it introduces dependencies among the edge payoffs. However, all our proofs (established in the edge-payoff case) can easily be reformulated in the vertex-payoff case. Examples in both frameworks are presented in Section \ref{sec:dimension1}.

    \subsection{Structural characterization}
    
    For this subsection, payoffs may be disregarded.~The transition structure given by the actions in $I$ naturally induces a directed graph on the state space.

    \begin{definition}[Transition graph]
        The \emph{transition graph} of the PMDP is a directed graph with vertex set $V = \bbZ^d$ and edge set $E = \{(\BFz, \BFz+\BFa) \mid \BFz \in \bbZ^d,\ \BFa \in I\}$.
    \end{definition}

    A key characteristic of the set $I$ is whether the induced transition graph contains cycles, which determines whether the token can return to a previously visited state. 

    \begin{definition}[Classification of PMDPs]\label{def:classification} A PMDP with action set $I \subset \bbZ^d$ is classified as:
        \begin{itemize}[noitemsep,topsep=0pt]
            \item \textit{trivial} if $I = \{\BFzero\}$ (the state never changes);
            \item \textit{oriented} if $\BFzero \notin \operatorname{conv}(I)$ (the transition graph is acyclic, so the token can never revisit states);
            \item \textit{controllable} if $\BFzero \in \operatorname{conv}(I)$ and $I \neq \{\BFzero\}$ (the transition graph contains cycles, allowing the token to revisit states without being restricted to do it).
        \end{itemize}
    \end{definition}

    In particular, the following geometric characterization holds.

    \begin{lemma}\label{lem:structural_characterization}
        Consider a PMDP  with action set $I \subset \bbZ^d$. 
        \begin{itemize}[noitemsep,topsep=0pt]
            \item If it is \textit{oriented}, there exists a direction $\BFu \in \bbR^d \setminus \{\BFzero\}$ such that $\BFa \cdot \BFu > 0$ for all $\BFa \in I$. In particular, the projection of the state onto $\BFu$ strictly increases at each step. 
            \item If it is \textit{controllable}, there exists a finite sequence of actions such that repeating it keeps the trajectory bounded.
        \end{itemize}
    \end{lemma}

    \begin{proof}
        With orientability, the claim follows by a separation theorem \citep[see][Corollary 11.4.2]{Rockafellar1997}.
        
        Let $I = \{\BFa_1, \ldots, \BFa_k\}$ with $k \geq 1$. If there is controllability, since $0 \in \operatorname{conv}(I)$, there exist $\alpha_i \geq 0$ with $\sum_{i=1}^k \alpha_i = 1$ such that $\sum_{i=1}^k \alpha_i \BFa_i = 0$. As $\BFa_i \in \bbZ^d$, $\alpha_i $ can be taken in $\bbQ_+$ and clearing denominators yields non-negative integers $(\lambda_i)$ not all zero, satisfying $\sum_{i=1}^k \lambda_i \BFa_i = 0$. This means that the sequence returns the token to its starting state, wherever that starting state is.  
    \end{proof}

    \begin{example}[Controllable]\label{ex:dimension2_controllable}
        Consider a PMDP on $\bbZ^2$ with $I \supset \{(1,0), (-2,-1), (1,1)\}$ and $\mathrm{Bernoulli}(p)$ payoffs with $p \in [0, 1]$. From every vertex of the transition graph there exists a same cycle (playing $(1,1)$ then $(-2, -1)$ and then $(1,0)$), as illustrated below:
        \begin{figure}[!hbt]
            \centering
            \begin{tikzpicture}[
                scale=0.8,
                baseline=(current bounding box.center),
                every node/.style={transform shape}
            ]
                % vertices
                \node[circle, draw, fill=black, inner sep=1.5pt,
                    label=right:{$\BFx+(1,1)$}] (RightUp) at (1.5,1.5) {};

                \node[circle, draw, fill=black, inner sep=1.5pt,
                    label=below:$\BFx$] (zero) at (0,0) {};

                \node[
                    circle, draw, fill=black, inner sep=1.5pt,
                    label={below:{$\BFx-(1,0)$}}
                ] (Left) at (-1.5,0) {};

                % moves
                \draw[->, thick, color=red, line width=0.8pt] (zero) -- (RightUp);
                \draw[->, thick, color=red, line width=0.8pt] (Left) -- (zero);
                \draw[->, thick, color=red, line width=0.8pt] (RightUp) to[bend right=20] (Left);
            \end{tikzpicture}
        \end{figure} 
    \end{example}
    
    \begin{example}[Oriented] \label{ex:oriented-percolation-on-Z2}
        Consider a PMDP on the square lattice with $I = \{\BFe_1, \BFe_2\}$ and $\mathrm{Bernoulli}(p)$ payoffs. This example is oriented and expands in all directions within $\bbR^2_{+}$. 
    \end{example}

    \subsection{Main Results}
    A central question in this model concerns the almost sure convergence of the sequence $(v_n(\BFz))_{n \geq 1}$, that is, the existence of a \emph{limit value}. A second question is whether this limit is deterministic. The foundational result in this regard is Theorem~2.1 in \citep{GarnierZiliotto2022}.
    The authors established the existence of the limit value for oriented percolation games, which intuitively corresponds to a class of games where the token moves in a fixed direction regardless of players' actions. Below, we state the PMDP version of their result.

    \begin{theorem}[1-player version of Theorem 2.1 and remark below in \citep{GarnierZiliotto2022}]\label{thm:iid_and_oriented_percolation_games}
        Consider an oriented PMDP and denote  $\|I\|_\infty=\max_{\BFa \in I}|\BFa|$. Then, for all $\BFz \in \bbZ^d$, the sequence $v_n(\BFz)$ converges $\Pbb$-almost surely to a constant $v_{\infty} \in \bbR$ as $n \to \infty$, independent of $\BFz$. Moreover, there exist constants $A = A(d, \|I\|_{\infty}, \|g\|_{\infty}) > 0$ and $B = (8\|g\|_{\infty}^2)^{-1}$ such that for all $n \geq 1$, $\BFz \in \bbZ^d$, and $\lambda \geq 0$,
        \[
            \Pbb\left(|v_n(\BFz) - v_{\infty}| \geq \lambda + A\sqrt{\frac{\ln(n + 1)}{n}} \right) \leq \exp(-B \lambda^2 n).
        \]
        In particular, $|v_{\infty} - \mathbb{E}v_n| \leq A\sqrt{\ln(n + 1)} n^{-1/2}$.
    \end{theorem}

    \begin{remark}\label{rem:constant}
    More precisely, it follows from the proof of Theorem 2.1 in \citep{GarnierZiliotto2022} that 
    \[
        |v_{\infty} - \mathbb{E}v_n|\leq C\|g\|_\infty (d+1)\sqrt{\ln(n\|I\|_\infty+1)}n^{-1/2}
    \]
    for some universal constant $C$. In particular, when $g$ and $d$ are fixed, 
    \[
        |v_{\infty} - \mathbb{E}v_n|\leq A\sqrt{\ln(n\|I\|_\infty+1)}n^{-1/2}
    \]
    for a positive constant $A$ that depends on $d$ and $g$ but not on $I$.
    \end{remark}

    \begin{remark}\label{rem:i.i.d.}
        In contrast to the 2-player setting, we assume that payoffs are i.i.d.\ across both states and actions. This differs from the definition of i.i.d.\ percolation games in \citep[Definition 2.2]{GarnierZiliotto2022}, where payoffs are required to be i.i.d~only across states. We impose this stronger assumption to establish the existence of the limit value without relying on the orientation assumption. And still, with this assumption the existence of the limit value for non-oriented percolation games is an open problem. Further discussion of the challenges that arise for PMDPs when payoffs are assumed to be i.i.d. only across states is provided in the Perspectives.
    \end{remark}

    The proof of Theorem~\ref{thm:iid_and_oriented_percolation_games} relies heavily on the orientation assumption, which prevents the token from revisiting states. In the 2-player setting, removing this assumption introduces significant difficulties, as shown in the analysis of a non-oriented percolation game in \citep{sepulveda2024gameorientedpercolation}. However, for PMDPs, the absence of an adversary simplifies the analysis. Our first result establishes the existence of the limit value for general sets of actions.

    \begin{theorem}\label{thm:limit-value}
        For $\Gamma(\BFz)$, the sequence $v_n(\BFz)$ converges $\Pbb$-almost surely to $v_{\infty}(\BFz)$ as $n \to \infty$. Moreover,
        \[
            v_{\infty}(\BFz) =
                \begin{cases}
                    g(\BFz, \BFzero) & \text{if } I = \{\BFzero\}, \\ 
                    M & \text{if there is controllability,}
                \end{cases}
        \]
        where $M$ denotes the essential supremum of $g$.
    \end{theorem}
    
    The second (and main) result of this paper establishes the existence of a strategy that guarantees the limit value exactly, which in particular implies the existence of the uniform value. Before stating this result, we introduce the notions of the uniform value and of an $\varepsilon$-optimal strategy. 

    \begin{definition}[uniform value and  $\varepsilon$-optimalilty, see Definition 2.3 and Claim 2.4 in \citep{Renault2011}]\label{def:uniform_value}
        For every $\BFz \in \bbZ^d$, we say that $\Gamma(\BFz)$ has a \emph{uniform value} if the limit value $v_{\infty}(\BFz)$ exists and 
        \[
            \forall \varepsilon > 0, \exists \sigma \in \Sigma \text{ such that } \exists n_0 \in \bbN \colon \ \forall n \geq n_0, \; \gamma_n(\BFz, \sigma) \geq v_{\infty}(\BFz) - \varepsilon, \quad \Pbb\text{-a.s.}
        \]
        In this case, we say that $v_{\infty}(\BFz)$ is the uniform value and for a fixed $\varepsilon > 0$, that $\sigma$ is \emph{$\varepsilon$-optimal}.
    \end{definition}

    The existence of a uniform value requires that, for every desired error $\varepsilon$, there exists a strategy whose performance remains $\varepsilon$-close to the limit value for all sufficiently large horizons. A classical example illustrating the distinction between these two notions of value for MDPs with countable state spaces is provided by \citep[Example in Section 2]{Lehrer1992}, for which the limit value exists, but the uniform value does not. 

    \begin{theorem}\label{thm:main-uniform}
        For $\Gamma(\BFz)$, there exists a 0-optimal strategy that attains the limit value, that is, 
        $$
            \exists \sigma \in \Sigma \colon \gamma_n(\BFz, \sigma) \to v_{\infty}(\BFz) \quad \text{as } \, n \to \infty \quad \Pbb\text{-a.s.}
        $$
        In particular, $\Gamma(\BFz)$ has a uniform value which, except for the trivial PMDP, is deterministic and independent of the initial state $\BFz$.
    \end{theorem}
    
    Our third result shows that, for oriented PMDPs, there exists a 0-optimal strategy that depends only on the token's current position and the realized payoffs, that is, a stationary 0-optimal strategy.

    \begin{definition}\label{def:stationary}
        A strategy $\sigma = (\sigma_m)_{m \geq 1} \in \Sigma$ is \emph{stationary} if there exists a mapping $\sigma^*: \Omega \times \bbZ^d \to  I$ such that for all $m \geq 1$, $\sigma_m(\omega, \BFz) = \sigma^*(\omega, \BFz_m)$, where $\BFz_m$ is the state at stage $m$ induced by $\sigma$ starting from $\BFz$.
    \end{definition}

    \begin{theorem}\label{thm:oriented-stationary}
        Every oriented PMDP admits a stationary 0-optimal strategy. 
    \end{theorem}
    
    Our final main result is the \emph{dimensional lifting property}, developed in Section \ref{sec:dimensional-lifting-oriented-percolation}, which establishes that high-dimensional PMDPs can be approximated via lower-dimensional ones.

\section{Proof of the main control-theoretic results}\label{sec:proof_main_results}

    \subsection{Existence of a limit value}

    First we prove Theorem~\ref{thm:limit-value}. The oriented case follows from Theorem~\ref{thm:iid_and_oriented_percolation_games}. For the controllable case, the key idea is to identify a reachable cycle whose edge-payoffs are arbitrarily close to their  essential supremum.

    \begin{proof}~[of Theorem~\ref{thm:limit-value}]
        We prove it in each of the three possible cases.
        
        If $I = \{\BFzero\}$, given an initial state $\BFz$, $\BFz_m = \BFz$ for all $m \geq 1$. Consequently, $v_{\infty}(\BFz) = g(\BFz, \BFzero)$. 
        
        If the PMDP is oriented it follows from Theorem~\ref{thm:iid_and_oriented_percolation_games}.
        
        If the PMDP is controllable, we fix $\varepsilon > 0$ and let $S_{\varepsilon}$ denote the class of cycles whose edge-payoffs are $\varepsilon$-close to their essential supremum, that is, for every $C \in S_{\varepsilon}$, it holds that 
        \begin{equation}\label{eq:desired_property}
            \text{for all } \BFz \in C, \;
            g(\BFz, \BFa_{i(\BFz)}) \geq M - \varepsilon,
        \end{equation}
        where $\BFa_{i(\BFz)}$ is the action played at $\BFz$ to keep the token within $C$.

        Let $\BFz \in \bbZ^d$ be the initial position. Define the $\varepsilon$-strategy $\tilde \sigma$ as follows: identify a state contained in a cycle $C$ in $S_{\varepsilon}$ and reachable from $\BFz$. Move there, and once there, choose actions such that the state remains within that cycle. Let $A_{\varepsilon}$ be the event that such a state exists. Next we show that it happens with probability 1.
        
        Since the PMDP is controllable and the transition graph is translation-invariant (because the actions set is state-independent), we can choose countably many pairwise disjoint translates of a fixed cycle that contain reachable points. Enumerate them. For each $m \geq 1$, let 
        $A_{\varepsilon, m} := \{m\text{-th cycle satisfies } \eqref{eq:desired_property}\}$. Using that the cycles are disjoint, that the payoffs are i.i.d. and the definition of the essential supremum, we obtain
        $$
            \Pbb(A_{\varepsilon}^c) \leq \Pbb\left(\bigcap_{m = 1}^{\infty}A_{\varepsilon, m}^c\right) = \prod_{m = 1}^{\infty}\Pbb(A_{\varepsilon, m}^c) = \lim_{m \to \infty}\left(1-\Pbb(g \geq M - \varepsilon)^{|C|}\right)^m = 0.
        $$
        Thus, for every initial state $\BFz$, almost surely, there exists $s_{\varepsilon}(\BFz) \geq 0$ such that a cycle in $S_{\varepsilon}$ can be reached in $s_{\varepsilon}(\BFz)$ steps. Consequently, since payoffs are uniformly bounded, almost surely
        \begin{equation}\label{eq:limit-value-controllabble}
            v_n(\BFz) \geq \frac{s_{\varepsilon}(\BFz)\operatorname*{ess\,inf}(g)}{n} + \frac{n - s_{\varepsilon}(\BFz)}{n}(M - \varepsilon) \geq M - \varepsilon + o(1), \quad \text{for all } n \ge s_{\varepsilon}(\BFz).
        \end{equation}
       Combining it with the bound $v_n \leq M$ and letting $\varepsilon \to 0$ yields that $v_n(\BFz) \to M$, as $n \to \infty$, almost surely.   \end{proof}

    \begin{remark} In the case of controlable PMDPs, by using the proposed strategy, $M$ itself may never be attained along any actual trajectory; that is, the strategy is not necessarily 0-optimal. For instance, if payoffs are $\mathrm{Unif}(0,1)$, then $g < 1$ almost surely, even though $\operatorname*{ess\,sup} g = 1$. The strategy is merely $\varepsilon$-optimal. With such a strategy, $M$ is attained if and only if $\mathbb{P}(g = M) > 0$.
    \end{remark}

    \subsection{Existence of a 0-optimal strategy}
    Now we prove Theorem~\ref{thm:main-uniform}. Its proof is mainly devoted to constructing a block-by-block strategy $\tilde \sigma$ that guarantees the limit value in the infinite-horizon average payoff with limit inferior criterion, that is,
        \begin{equation}\label{eq:strategy-blocks}
            \liminf_{n \to \infty} \gamma_n(\BFz, \tilde \sigma) \geq v_{\infty}(\BFz).
        \end{equation}
        Once this is done, Theorem~\ref{thm:main-uniform} follows immediately: see \eqref{eq:conclusion}.

    The block-by-block strategy consists in playing a $k$-stage optimal strategy at block $k$. Then, achieving the limit value relies on an asymptotic equicontinuity property of $(v_n(\BFz))$ that we prove next as a lemma. It asserts that convergence to the limit value is uniform over the growing set of states reachable within a sequence of blocks, allowing the initial position to drift with the stage index.
    
    \begin{definition}[Reachable set]\label{def:reachable-set}
        For $\BFz \in \bbZ^d$ and $n \ge 1$, the set of reachable states from $\BFz$ in $n$ stages is denoted by
        \[
            \mathcal{R}_n(\BFz) := \BFz + nI = \left\{ \BFz + \sum_{k=1}^n \BFa_k : \BFa_k \in I \right\}, \quad \text{with} \quad \mathcal{R}_n := \mathcal{R}_n(\BFzero)
        \]
        % When $\BFz = 0$, we simply write $\mathcal{R}_n$. 
    \end{definition}

    \begin{lemma}\label{lem:equicontinuity}
        For every PMDP, every initial state $\BFz \in \bbZ^d$, and every $k \in \bbN$, it holds that
        \begin{equation}\label{eq:casi-uniform}
            \lim_{n\to\infty} \sup_{\BFz' \in \mathcal{R}_{n^k}(\BFz)} |v_n(\BFz') - v_{\infty}(\BFz)| = 0, \quad \Pbb\text{-a.s.}
        \end{equation}
    \end{lemma}
    
    \begin{proof} 
        Fix $\varepsilon > 0$. 
        If $I = \{\BFzero\}$ the result holds trivially.
        
        If $I \neq \{\BFzero\}$, the reachable set $\mathcal{R}_{n^k}(\BFz)$ is contained in the intersection of $\bbZ^d$ and a ball of radius $n^k \|I\|_{\infty}$ with $I$ finite. Thus, for all $\varepsilon > 0$, from the union bound we get
        \[
            \Pbb\left(\sup_{\BFz' \in \mathcal{R}_{n^k}(\BFz)} |v_n(\BFz') - v_{\infty}| > \varepsilon\right) \le \left|\mathcal{R}_{n^k}(\BFz)\right| \Pbb(|v_n(\BFz') - v_{\infty}| > \varepsilon)
        \]
        The desired almost sure convergence would follow from the Borel-Cantelli lemma if the probability $\Pbb(|v_n(\BFz') - v_{\infty}| > \varepsilon)$ is summable in $n$, uniformly in $\BFz' \in \mathcal{R}_{n^k}(\BFz)$. Next we prove that this is the case. 
        
        For oriented PMDPs, the concentration estimates from Theorem~\ref{thm:iid_and_oriented_percolation_games} give that there exists $n_0 \geq 0$ such that for all $n \geq n_0$ and all $\BFz' \in \bbZ^d$,
        \[
            \Pbb(|v_n(\BFz') - v_{\infty}| \geq \varepsilon) \leq \exp\left(-\frac{B\varepsilon^2n}{4}\right),
        \]
        which is summable in $n$, uniformly in $\BFz'$.

        For controllable PMDPs, $v_{\infty} = M$. Assume $\varepsilon < M - m$, where $m$ and $M$ are the essential infimum and supremum of the payoffs, respectively (if $\varepsilon \geq M - m$, the probability is trivially 0 for large $n$). Choose a fixed $\delta \in (0, \varepsilon)$. From \eqref{eq:limit-value-controllabble}, we have the lower bound
        \[
            v_n(\BFz') \geq M - \delta - \frac{s_{\delta}(\BFz')}{n}(M - m).
        \]
        Hence, if $M - v_n(\BFz') \geq \varepsilon$, then $s_{\delta}(\BFz') \geq n \frac{\varepsilon - \delta}{M - m}$. Thus,
        \[
            \Pbb(M - v_n(\BFz') \geq \varepsilon) \leq \Pbb\left(s_{\delta}(\BFz') \geq n \frac{\varepsilon - \delta}{M - m}\right).
        \]
        Within $N$ steps, the transition graph contains at least $\lfloor N / |C| \rfloor$ pairwise disjoint cycles of length $|C|$, where $|C| \ge 1$ is the number of points in the smallest cycle. Recall the definition of $S_{\delta}$ by \eqref{eq:desired_property} with $\varepsilon = \delta$. The event $\{s_{\delta}(\BFz') \geq N\}$ implies that none of these disjoint cycles belong to $S_{\delta}$. Since the payoffs are i.i.d., the probability of this event is bounded by
        \[
            \Pbb(s_{\delta}(\BFz') \geq N) \leq \left(1 - \Pbb(g \geq M - \delta)^{|C|}\right)^{\lfloor N / |C| \rfloor} \leq \exp\left(- \frac{N\Pbb(g \geq M - \delta)^{|C|}}{|C|} \right).
        \]
        Substituting $N = n \frac{\varepsilon - \delta}{M - m} > 0$, we obtain
        \[
            \Pbb(M - v_n(\BFz') \geq \varepsilon) \leq \exp\left(- \frac{(\varepsilon - \delta) \Pbb(g \geq M - \delta)^{|C|}}{|C|(M - m)} n \right).
        \]
        Since $\delta < \varepsilon$ and $\Pbb(g \geq M - \delta) > 0$ by definition of the essential supremum, the exponent is strictly negative and linear in $n$. Thus, the probability is summable in $n$, uniformly in $\BFz'$.  \end{proof}

        \begin{proof}~[of Theorem~\ref{thm:main-uniform}]
        If $I = \{\BFzero\}$, then $\mathcal{R}_n(\BFz) = \{\BFz\}$ for all $n$, and the result holds trivially. Assume henceforth that $I \neq \{\BFzero\}$. For each $\BFz \in \bbZ^d$, let $\Omega_{\BFz}$ be the full-measure set on which $v_k(\BFz)$ converges to $v_{\infty}$ almost surely as $k \to \infty$ (guaranteed by Theorem~\ref{thm:limit-value}), and let $\tilde\Omega$ be the full-measure set on which \eqref{eq:casi-uniform} holds. Since $\bbZ^d$ is countable, the intersection
        \[
            \Omega^* := \bigcap_{\BFz\in\bbZ^d}\Omega_{\BFz}\cap\tilde\Omega
        \]
        also has probability one. Fix $\omega \in \Omega^*$ and $\BFz\in\bbZ^d$. Consider the block strategy $\tilde{\sigma}$ defined as follows: play an optimal $k$-stage strategy $\sigma^{(k)}$ during block $k$, for $k=1,2,\dots$. Let $T_n = \sum_{k=1}^n k = n(n+1)/2$. Then, 
        \[
            \gamma_{T_n}^{\omega}(\BFz, \tilde{\sigma}) = \frac{1}{T_n}\sum_{k=1}^n\sum_{m=0}^{k-1} g_\omega\bigl(\BFz_{k,m},\sigma^{(k)}(\BFz_{k,m})\bigr)
            =
            \frac{1}{T_n}\sum_{k=1}^n k\, v_k^\omega(\BFz_{k-1,0}),
        \]
        where $\BFz_{k-1,0}$ is the state at which the token enters the $k$-th block and $\BFz_{0,0}=\BFz$. It holds that $\BFz_{k-1,0} \in R_{(k-1)k/2}$ and Lemma~\ref{lem:equicontinuity} thus applies, yielding $v_k^\omega(\BFz_{k-1,0}) \to v_{\infty}$ as $k \to \infty$. By the weighted Cesàro mean theorem, since $v_{k}^{\omega}(\BFz_{k-1,0}) \to v_{\infty}$, and the weights satisfy $T_n = \sum_{k = 1}^nk \to \infty$, we have
        \begin{equation}\label{eq:Tnconvergence}
            \lim_{n\to\infty}\gamma_{T_n}^{\omega}(\BFz, \tilde{\sigma}) = \lim_{n \to \infty}\frac{1}{T_n}\sum_{k=1}^n k\, v_k^\omega(\BFz_{k-1,0})=v_{\infty}.
        \end{equation}
        To extend this result to arbitrary horizons $N \in \mathbb{N}$, choose $n$ such that $T_n \le N < T_{n+1}$.We have
        \[
            T_n\gamma_N^{\omega}(\BFz, \tilde{\sigma})+(n+1)\|g\|_\infty \geq N\gamma_N^{\omega}(\BFz, \tilde{\sigma}) \ge T_n\gamma_{T_n}^{\omega}(\BFz, \tilde{\sigma}) - (n+1)\|g\|_\infty.
        \]
        Dividing by $T_n$, 
        \[
            \gamma_N^{\omega}(\BFz, \tilde{\sigma}) \ge \gamma_{T_n}^{\omega}(\BFz, \tilde{\sigma}) - \frac{2\|g\|_\infty(n+1)}{T_n}.
        \]
        As $N \to \infty$, we have $n \to \infty$, and thus $(n+1)/T_n \to 0$. From this and \eqref{eq:Tnconvergence} we get
        \begin{equation}\label{eq:block-oriented}
            \liminf_{N\to\infty} \gamma_{N}^{\omega}(\BFz, \tilde \sigma) \ge v_{\infty}.
        \end{equation}
        From \eqref{eq:block-oriented}, the definition of $v_n$ and Theorem~\ref{thm:limit-value}, for all $\BFz \in \bbZ^d$, we obtain
        \begin{equation}\label{eq:conclusion}
            v_{\infty}(\BFz) \leq \liminf_{n \to \infty}\gamma_n(\BFz, \tilde \sigma) \leq \limsup_{n \to \infty}\gamma_n(\BFz, \tilde \sigma) \leq \limsup_{n \to \infty} v_n(\BFz) = v_{\infty}(\BFz), \quad \Pbb\text{-a.s.}
        \end{equation}
        Consequently, $\tilde \sigma$ is optimal. By definition of limit superior and inferior, we obtain that the uniform value exists. In particular, by Theorem~\ref{thm:limit-value}, the uniform value is deterministic and independent of the initial state for oriented and controllable PMDPs.  \end{proof}

    \begin{remark}\label{rem:extension-2-players}
        The dynamic programming core of the block-by-block construction of the strategy in the oriented case is that the concentration results inherited from oriented percolation games yields uniform convergence on a reachable set whose size is polynomial, allowing the optimal strategies of finite-stage games to be concatenated. These ideas naturally extend to the zero-sum two-player framework. Thus, not only $v_n$ converges a.s. (as proven in \citep[Theorem 2.1]{GarnierZiliotto2022}), but it can also be established that there exist a uniform value and $0$-optimal strategies. A full proof would require a precise review of the two-player model, so we omit it here.
    \end{remark}

    \subsection{Existence of a $0$-optimal stationary strategy for oriented PMDPs}

    In Theorem~\ref{thm:main-uniform} we proved that for every initial state $\BFz$ there exists a (possibly time-dependent) optimal strategy $\sigma = (\sigma_{m})_{m \geq 1}$ in $\Gamma(\BFz)$. We now construct for oriented PMDPs a stationary optimal strategy $\sigma^*$ from this family, thus proving Theorem~\ref{thm:oriented-stationary}.

    \begin{proof}~[of Theorem~\ref{thm:oriented-stationary}]
        Fix $\omega \in \Omega$. For any starting state $\BFz \in \bbZ^d$, and stage number $m \geq 1$, let $f_{\BFz, m} \colon \Omega \to \bbZ^d$ be the state at stage $m$ when using strategy $\sigma$ in $\Gamma(\BFz)$. That is, following $\sigma$ from $\BFz$ generates a trajectory 
        $$
            f_{\BFz, 1}(\omega) = \BFz, f_{\BFz, 2}, f_{\BFz, 3}, \ldots.
        $$ 
        Fix a bijection $\phi:\bbN\to\bbZ^d$ (for better intuition, one can order $\bbZ^d$ by the Euclidean distance from the origin, breaking ties lexicographically) and define $r_{\BFz}: \Omega \to \bbN$ by
        \[
            r_{\BFz}(\omega) := \min\{n\in\bbN\} \colon \text{there exists } m \geq 1, \, \text{such that } \BFz = f_{\phi(n), m}(\omega)\}.
        \]
        Thus, $\phi(r_{\BFz}(\omega))$ is the smallest (first in the enumeration $\phi$) $\BFz'$ such that $\BFz$ appears on the optimal trajectory starting from $\BFz'$. 
        \begin{figure}[!h]
            \centering
            \begin{tikzpicture}[scale=0.4]
                \coordinate (O) at (0,0);
                \fill (O) circle (2.5pt);
                \node[below=1pt] at (O) {\scriptsize $0$};

                % dashed arc centered at 0
                \draw[dashed, thick] (-0.8, 1.5) arc (118:-20:1.7);

                % point z1 on arc
                \coordinate (z1) at (1.3, 0.8);	
                \node[below=2pt] at (z1) {\scriptsize $\BFz_1$};

                % point z2
                \coordinate (z2) at (1.1, 4);
                \node[left=2pt] at (z2) {\scriptsize $\BFz_2$};

                % intersection point z
                \coordinate (z) at (4, 4.5);
                \node[below right] at (z) {\scriptsize $\BFz$};

                % point above z on blue path
                \coordinate (zp) at (5.1, 5.8);
                \node[right] at (zp) {\scriptsize $\BFz + \sigma^*(\omega, \BFz)$};
                \draw (zp) circle (3pt);

                \draw [blue!40!purple!80!black, xshift=4cm]
                    (z1) to[out=50,in=240] (z);

                \draw [magenta!40, line width=4pt, opacity=0.5, xshift=4cm]
                    (z) to[out=50,in=240] (zp);
                \draw [blue!40!purple!80!black, xshift=4cm]
                    (z) to[out=50,in=240] (zp);

                \draw [blue!40!purple!80!black, xshift=4cm]
                    (zp) to[out=50,in=240] (4.5, 8.4);

                \draw [black!70!green, xshift=4cm] (z2) to[out=50,in=240] (z);
                \draw [black!70!green, xshift=4cm] (z) to[out=50,in=240] (5, 4.5);
                \draw [black!70!green, xshift=4cm] (5, 4.5) to[out=50,in=240] (7, 6);

                \fill (z1) circle (4pt);
                \fill (z2) circle (4pt);
                \fill (zp) circle (4pt);
                \fill (z) circle (4pt);

            \end{tikzpicture}
            \caption{Two optimal strategies pass through the state $\BFz$. We assign them a priority order by their Euclidean distance to the origin. The stationary strategy $\sigma^*$ in $\BFz$ follows the 0-optimal trajectory with the smallest index in the priority order, shown in violet.}
        \end{figure}
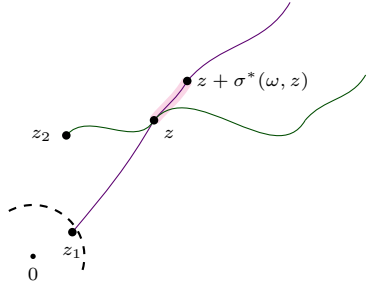
        Let also $m_{\BFz}(\omega)$ be the unique stage such that $\BFz = f_{\phi(r_{\BFz}(\omega)),m_{\BFz}(\omega)}$; the uniqueness of $m_{\BFz}(\omega)$ immediately follows from the fact that the PMDP is oriented.
        
        We now define $\sigma^*$ by $\sigma^*(\omega,\BFz) := \sigma_{m_{\BFz}(\omega)}(\omega, \phi(r_{\BFz}(\omega)))$. Thus, $\sigma^*$ chooses, for any state $\BFz$, among all optimal trajectories passing through $\BFz$, the one that originates from the smallest (for the enumeration $\phi$) starting state $\BFz'$, and the decision maker is asked to play as in this trajectory. 

        Now, since $\BFz = f_{\phi(r_{\BFz}(\omega)), m_{\BFz}(\omega)}(\omega)$, one has 
        $$
            \BFz + \sigma^*(\omega,\BFz) = f_{\phi(r_{\BFz}(\omega)), m_{\BFz}(\omega) + 1}(\omega) \; \text{ and } \; r_{\BFz + \sigma^*(\omega, \BFz)}(\omega) \leq r_{\BFz}(\omega).
        $$ 
        Thus, for any $\omega$ and starting state $\BFz$, the index $r$ is non-increasing on the trajectory followed by $\sigma^*$ starting from $\BFz$. In particular, it is equal to some $\BFa(\omega, \BFz)$ from some stage on (it cannot decrease forever). Once $r$ stabilizes, the strategy $\sigma^*$ stops switching trajectories and permanently follows the single trajectory originating from $\phi(\BFa(\omega, \BFz))$, which is optimal by assumption.  \end{proof}

    \begin{remark}
        The key mechanism of the proof is independent of the percolation configuration. Central elements are that the graph is acyclic and that from every state there is a $0$-optimal strategy. With it, an ordering argument allows the decision maker to select a stationary $0$-optimal strategy.
    \end{remark} 

    \begin{remark}\label{rem:no0stationary} 
        For controllable PMDPs, the strategy constructed in the proof of Theorem~\ref{thm:limit-value} is stationary but merely $\varepsilon$-optimal, while the one constructed in the proof of Theorem~\ref{thm:main-uniform} is optimal but not stationary. In fact, a stationary optimal strategy is not guaranteed to exist. For instance, consider a one-dimensional PMDP with $I = \{0,1\}$ and payoffs drawn from a continuous distribution, e.g. $\mathrm{Unif}(0,1)$. A stationary strategy will either always move right by playing $+1$, or eventually choose $0$ at some state $z$ and remain there forever. The former scenario gives a limit value of 1/2 by the strong law of large numbers (SLLN), while the latter yields a payoff of $g(z) < 1, \ \Pbb$-a.s.
    \end{remark}

\section{Dimensional lifting and oriented percolation}\label{sec:dimensional-lifting-oriented-percolation}

    Consider a PMDP $\Gamma^{1,m}_F$ in $\bbZ$ with action set $I = \{1,m\}$, $m \ge 1$ and payoff distribution $F$.~In Theorem~\ref{thm:convergence_game_m}, we prove that the sequence of limit values of the one-dimensional PMDPs $\Gamma^{1,m}_F$ converges, as $m \to \infty$, to the limit value of the two-dimensional PMDP $\Gamma^{(2)}_F$ defined below. 
    \begin{example}\label{ex:general-oriented-percolation}
        A PMDP with action set $I = \{\BFe_1,\BFe_2\}$ and payoff distribution $F$.
    \end{example}
    \noindent Intuitively, when $m$ is large, choosing action $m$ transitions the PMDP to a distant, approximately independent region of the environment. Alternatively, choosing $m$ switches the PMDP to a separate parallel copy of the integer line, naturally leading to the higher-dimensional PMDP. It is this phenomenon that we refer to as the dimensional lifting property. For clarity, we prove it for that specific sequence of PMDPs. Then, we discuss possible extensions in Remark~\ref{rem:extensions-lifting}.
    
    In the subsequent section, Proposition~\ref{prop:characterization_critical_proba} characterizes the critical parameter of oriented Bernoulli percolation with the uniform value of Example \ref{ex:oriented-percolation-on-Z2} (which is $\Gamma_{F}^{(2)}$ with Bernoulli payoffs). As a baseline for this result, we first prove a continuity result for PMDPs with Bernoulli payoffs. Next, we present numerical simulations leveraging dimensional lifting to illustrate the phase transition obtained by varying the step size $m$ in the sequence $\Gamma^{1,m}_F$ with Bernoulli payoffs in Proposition \ref{prop:bernoulli-continuity}. Lastly, we compare three related examples.

    \subsection{Dimensional lifting}
    \begin{theorem}\label{thm:convergence_game_m}
        Let $v_{\infty}^{1,m}$ denote the limit value of $\Gamma^{1, m}_F$ and $v_{\infty}^{(2)}$ the one of $\Gamma^{(2)}_F$. Then,
        \[
            v_{\infty}^{1,m} \rightarrow v_{\infty}^{(2)} \quad \text{as } \, m \to \infty \quad \Pbb\text{-a.s.}
        \]
    \end{theorem}

    \begin{proof}
        Fix $F$, which we omit for the sake of readability. Let us construct a coupling between $\Gamma^{1, m}$ and $\Gamma^{(2)}$ over the same probability space on which the payoffs are defined. Let $\phi_m \colon \bbZ \to \{0, 1, \dots, m-1\} \times \bbZ$ be defined by $\phi_m(z) = \left(z \bmod m, \left\lfloor \frac{z}{m} \right\rfloor\right)$. For the coupling to be well-defined, we require $m > n$. This ensures that within $n$ stages, there are no additional identifications of distinct two-dimensional states within the relevant finite-horizon region. Under this condition, the finite reachable subgraph in $\bbZ$ is isomorphic to the finite reachable subgraph in $\bbZ^2$. 
        \begin{figure}[!ht]
            \centering
            \begin{tikzpicture}[scale=0.8, every node/.style={transform shape}]

                \draw[gray!30] (0,1) grid (3,3);
                \draw[->, thick, blue] (0,1) -- (1,1);
                \filldraw[black] (1,1) circle (2pt) node[below right] {$\phi_m(z) + \BFe_1$};
                \draw[->, thick, red] (0,1) -- (0,2) node[midway, left] {$g(\phi_m(z), \BFe_2)$};
                \filldraw[black] (0,2) circle (2pt) node[above left] {$\phi_m(z) + \BFe_2$};
                \filldraw[black] (0,1) circle (2pt) node[below left] {$\phi_m(z)$};

                \draw[<->, dashed, thick, black] (1.5, 0.1) -- (1.5, -0.8) node[midway, right] {$\;$ Equivalent $n$-stage decision process under coupling with $m > n$.};

                \draw[thick, ->, gray!30] (0,-2.3) -- (3,-2.3) node[right] {$\dots$};
                \filldraw[black] (0,-2.3) circle (2pt) node[below left] {$z$};
                \draw[->, thick, blue] (0,-2.3) -- (1,-2.3);
                % node[midway, below] {$g(z, 1)$};
                \filldraw[black] (1,-2.3) circle (2pt) node[below right] {$z + 1$};
                
                \draw[thick, ->, gray!30] (0,-1.1) -- (3,-1.1) node[right] {$\dots$};
                \filldraw[black] (0,-1.1) circle (2pt) node[above left] {$z + m$};
                
                \draw[->, thick, red] (0,-2.3) .. controls (-0.4,-1.7) .. (0,-1.1) node[midway, left] {$g(z, m)$};

            \end{tikzpicture}
            \caption{A representation of the coupling between $\Gamma^{1,m}$ and $\Gamma^{(2)}$. After rescaling the vertical coordinates in $\bbZ^2$ by $m$, the two allowed moves become like two directions in a two-dimensional oriented lattice.}
        \end{figure}

        We couple the underlying random variables as follows: for each $z \in \bbZ$, the payoff on the edge $\{z, z+1\}$ is identical to that on $\{\phi_m(z), \phi_m(z) + \BFe_1\}$ in $\bbZ^2$, and the payoff on $\{z, z+m\}$ in $\bbZ$ is identical to that on $\{\phi_m(z), \phi_m(z) + \BFe_2\}$ in $\bbZ^2$. We then assign i.i.d. payoffs to the remaining edges on the square lattice.
        
        Consequently, under the coupling we have:
        \begin{equation*}
            v_{n}^{1,m}(z) = v_n^{(2)}(\phi_m(z)), \quad \text{for all } z \in m\bbZ \; \text{ and } \; n < m.
        \end{equation*}
        In particular, $\E v_{n}^{1, n+1} = \E v_n^{(2)}$ for all $n \geq 1$. From Theorem~\ref{thm:iid_and_oriented_percolation_games} and Remark~\ref{rem:constant}, we have the estimates 
        \begin{equation*}
            |\E v_{n}^{1,n+1} - v_{\infty}^{1, n+1}| \leq A \sqrt{\ln(n(n+1) + 1)}n^{-1/2} \text{ and }
            |\E v_n^{(2)} - v_{\infty}^{(2)}| \leq A \sqrt{\ln(n + 1)}n^{-1/2}		
        \end{equation*}
        where $A$ does not depend on $n$. Consequently, there exists a constant $C > 0$ such that
        \[
            \left|v_{\infty}^{1, n+1} - v_{\infty}^{(2)}\right| \leq \left|v_{\infty}^{1, n+1} - \E v_{n}^{1,n+1}\right| + \left|\E v_n^{(2)} - v_{\infty}^{(2)}\right| < C \sqrt{\ln n }\ n^{-1/2}.		
        \]
         \end{proof}

    \begin{remark}\label{rem:extensions-lifting}
        The established convergence phenomenon can be generalized in several ways. It can be reformulated for the vertex-based payoff setting. Furthermore, let $v_{\infty}^I$ denote the limit value of an oriented PMDP with action set $I \subset \bbZ^d$. Consider a sequence of PMDPs with action sets $I_m = I \cup \{m \cdot \BFone_d\}$. Then, a similar coupling argument yields
        \[
            v_{\infty}^{I_m} \xrightarrow[m \to \infty]{} v_{\infty}^{I'}, \text{ where } I' = \{(\BFa, 0) : \BFa \in I\} \cup \{\BFe_{d+1}\} \subset \bbZ^{d+1}.
        \]
        Specifically, with $I = \{\BFa_1, \BFa_2\} \subset \bbZ^2$, introducing a third action of magnitude $m$ and letting $m \to \infty$ yields convergence to the limit value of a corresponding decision process in $\bbZ^3$. It can also be extended by considering the sequence of action sets in dimension one $(\{1, m, m^2\})_{m \ge 1}$. Letting $m \to \infty$ then yields a three-dimensional PMDP. 

    \end{remark}
        
    \subsection{Characterization of the critical threshold}\label{subsec:critical_probability_threshold}
    
    Recall Example \ref{ex:oriented-percolation-on-Z2}. If an infinite oriented open path exists in the reachable quadrant (i.e., a sequence $(\BFx_k)_{k \in \bbN}$ such that $\BFx_{k+1} - \BFx_k \in \{\BFe_1, \BFe_2\}$, $g(\BFx_k, \BFx_{k+1}-\BFx_k) = 1$ and $\BFx_k \in \bbN^2$ for all $k \in \bbN$), the decision maker can reach it in finitely many steps and then follow it forever. Since the number of steps required to reach it does not affect the asymptotic average payoff, this yields a limit value of 1. In this subsection we prove that the limit value is 1 almost surely if and only if the Bernoulli parameter $p$ lies above the critical probability threshold for the emergence of such an infinite oriented path.
    
    Consider oriented Bernoulli percolation on $\bbZ_+ \times \bbZ_+$ with directed edges running from $\BFx$ to $\BFx+\BFe_1$ and $\BFx+\BFe_2$, and where each edge is declared open (payoff 1) with probability $p$ and closed (payoff 0) with probability $1-p$, independently of all other edges \citep{Durrett1984, Grimmett2002}. The critical probability threshold is defined as
    \[
        p_c := \inf \bigl\{p \in [0,1] : \Pbb_p(\text{there exists an infinite oriented open path}) = 1 \bigr\},
    \]
    where $\Pbb_p$ is the product probability measure associated with the independent edge payoffs and Bernoulli parameter $p$. This critical parameter is known to be non-trivial, i.e., $p_c \in (0,1)$ \citep[see][Sections 3 and 6]{Durrett1984}. We consider Bernoulli percolation for all parameters $p \in [0,1]$ simultaneously on a common probability space via the classic increasing coupling \citep[see][Proposition 2.1]{DuminilCopin2018}, with $\mathbb{P}$ the coupling probability measure. 
    
    The characterization of the critical parameter reads as follows.
    \begin{proposition}\label{prop:characterization_critical_proba}
        The uniform value $p \mapsto v_\infty(p)$ of the PMDP in Example \ref{ex:oriented-percolation-on-Z2} satisfies
        \[
            v_{\infty}(p) = 1 \quad \mathbb{P}\text{-a.s.} \quad \text{if and only if} \quad p \ge p_c.
        \]
    \end{proposition}
    This result is closely related in spirit to \citep[Theorem 5.1]{sepulveda2024gameorientedpercolation}. There, the authors characterize the same critical parameter using a (two-player) percolation game. For proving it, we will use the following general continuity result of the uniform value with respect to the Bernoulli parameter.
    \begin{proposition}\label{prop:bernoulli-continuity}
        Consider a non-trivial PMDP with $\text{Bernoulli}(p)$ payoffs on $\bbZ^d$. Then, the limit value function $p \mapsto v_{\infty}(p)$ is Lipschitz on $[p_0, 1]$ for every $p_0 > 0$. Moreover, if it is oriented, then the value function is also continuous at $p = 0$.
    \end{proposition}
    \begin{proof}~[of Proposition \ref{prop:bernoulli-continuity}]
        We first prove continuity on $(0, 1]$. For $0 < p < p' \leq 1$, we claim that
        \begin{equation}\label{eq:continuity_inequalities}
            v_{\infty}(p) \leq v_{\infty}(p') \leq \frac{p'}{p}v_{\infty}(p).
        \end{equation}
        Together, these establish continuity on $(0, 1]$. 
        The first inequality follows from the monotonicity of $v_{\infty}(p)$, which holds since the family of measures $\Pbb_{p \in [0,1]}$ are coupled in an increasing way. For the second inequality, observe that a $\text{Bernoulli}(p')$ variable thinned by a factor $p/p'$ (keep each success independently with probability $p/p'$, and otherwise turn it into 0) is distributed as $\text{Bernoulli}(p)$. Applying this to a configuration $\omega'$ sampled from $\Pbb_{p'}$ gives a configuration $\omega$ with law $\Pbb_p$, yielding
        \[
            \E[v_{n}(p')] \leq \frac{p'}{p}\E[v_{n}(p)].
        \]
        Passing to the limit as $n \to \infty$, the Dominated Convergence Theorem gives the claimed inequality.

        Let $p_0 > 0$, and consider any $p, p' \in [p_0, 1]$. Rearranging the second inequality in \eqref{eq:continuity_inequalities}, and given that the payoffs are bounded by 1, gives
        \[
            0 \leq v_{\infty}(p') - v_{\infty}(p) \leq \left(\frac{p'}{p} - 1\right)v_{\infty}(p) \leq \frac{p' - p}{p_0},
        \]
        which concludes the proof of the first statement of the proposition.

        To establish continuity at $p = 0$ for the oriented case, based on the estimates in Theorem~\ref{thm:iid_and_oriented_percolation_games}, it suffices to show that $\E[v_{n}(p)]$ can be made arbitrarily small by choosing $p$ small and $n$ large. Using Jensen's inequality, the definition of the $n$-stage decision process value and the positivity and monotonicity of the exponential function gives
        \[
            \exp\left(\lambda n \E[v_n(p)]\right) \leq \E\left[\exp\left(\lambda n v_n(p)\right)\right] \leq \sum_{\pi} \E\left[\exp\left(\lambda \sum_{m=1}^n c(e^{\pi}_m)\right)\right],
        \]
        where the sum runs over all paths $\pi$ of length $n$ starting at $z$, and $e_m^\pi$ is the edge taken at stage $m$. Since the payoffs are i.i.d. $\text{Bernoulli}(p)$ random variables, we obtain
        \begin{equation*}
            \E[v_n(p)] \leq \frac{1}{\lambda n}\ln \left( \sum_{\pi} \E\left[\exp\left(\lambda \sum_{m=1}^n c(e^{\pi}_m)\right)\right] \right) = \frac{1}{\lambda}\left(\ln |I| + \ln(1 - p + pe^{\lambda})\right).
        \end{equation*}
        Choosing $\lambda = -\ln p$ yields
        \[
            \E[v_n(p)] \leq \frac{C}{-\ln p},
        \]
        for some constant $C$ independent of $n$ and $p$.  \end{proof}

    \begin{remark} The dependence on the essential supremum could induce discontinuities in the limit value with respect to the payoff distribution. For example, consider a controllable PMDP with $\mathrm{Bernoulli}(p)$ payoffs. The limit value $p \mapsto v_\infty(p)$ is discontinuous at $p = 0$, since by Theorem~\ref{thm:limit-value},
        \[
            v_\infty(p) = M(p) = \begin{cases}
                    0 & \text{if } p = 0,\\
                    1 & \text{if } p > 0.
                \end{cases}
        \]
    \end{remark}

    \begin{proof}~[of Proposition~\ref{prop:characterization_critical_proba}]
        From the discussion at the beginning of this section, for $p > p_c$, $v_{\infty}(p) = 1$ almost surely. For $p = p_c$, Proposition~\ref{prop:bernoulli-continuity} implies that $v_{\infty}(p_c) = 1$ almost surely as well. If $p < p_c$, Claim 5.2 in \citep{sepulveda2024gameorientedpercolation} gives that, almost surely, the supremum over all semi-infinite oriented paths of the asymptotic average payoff is strictly less than 1, giving that $v_{\infty}(p) < 1$.  \end{proof}

    \subsection{Simulations}\label{subsec:simulations}
     From Theorem~\ref{thm:convergence_game_m} and Proposition~\ref{prop:characterization_critical_proba}, there is a sequence of one-dimensional PMDPs that approximate a two-dimensional PMDP, whose value function exhibits a phase transition at the critical threshold of oriented Bernoulli percolation. Available estimates for this threshold \citep[Section 4.1]{Jensen1999} indicate that $p_c \approx 0.644700185$.
    
    Figure \ref{fig:simulations1} illustrates the finite-horizon value curves for an increasing sequence of these one-dimensional PMDPs. We use these curves to track their approximate phase transition and compare it with the estimate of $p_c$ from \citep[Section 4.1]{Jensen1999}. 
    
    \begin{figure}[!hbt]
        \centering
        \includegraphics[width=0.5\linewidth]{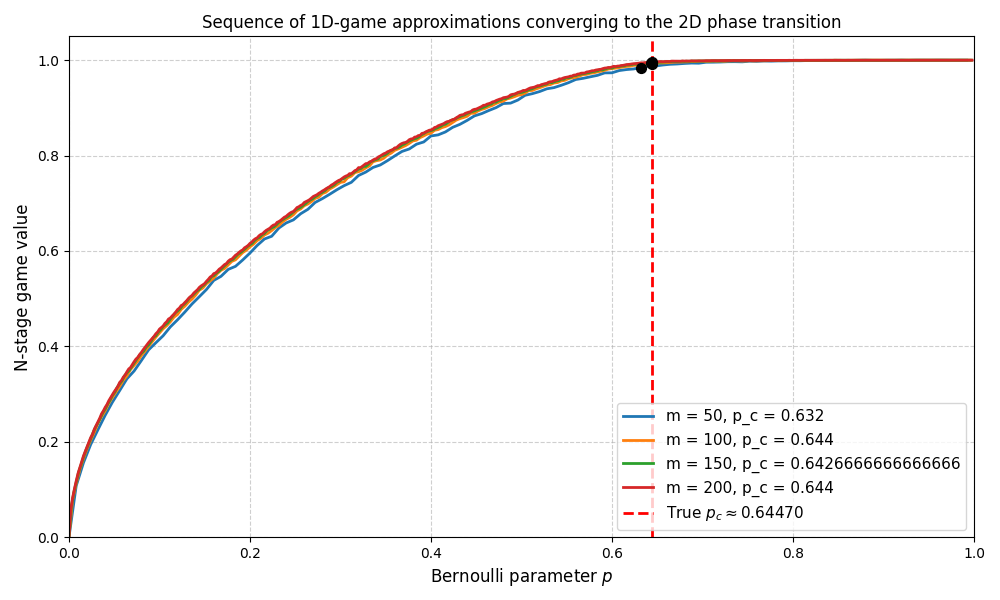}
        \caption{For increasing values of the parameter $m$, the value of the one-dimensional PMDPs with finite horizon $N=m-1$, action set $\{1,m\}$, and Bernoulli$(p)$ payoffs is computed over a range of equally spaced values of $p$ using Monte Carlo simulations and dynamic programming. To extract the approximate critical threshold from the curves, we identify the first point that comes within a small tolerance $\varepsilon$ of its maximum possible value, i.e., $v_N^{1,m}(p) \geq 1 - \varepsilon$, and for which all subsequent points also satisfy this condition.}
        \label{fig:simulations1}
    \end{figure}	

    \begin{figure}[!hbt]
        \centering
        \includegraphics[width=0.5\linewidth]{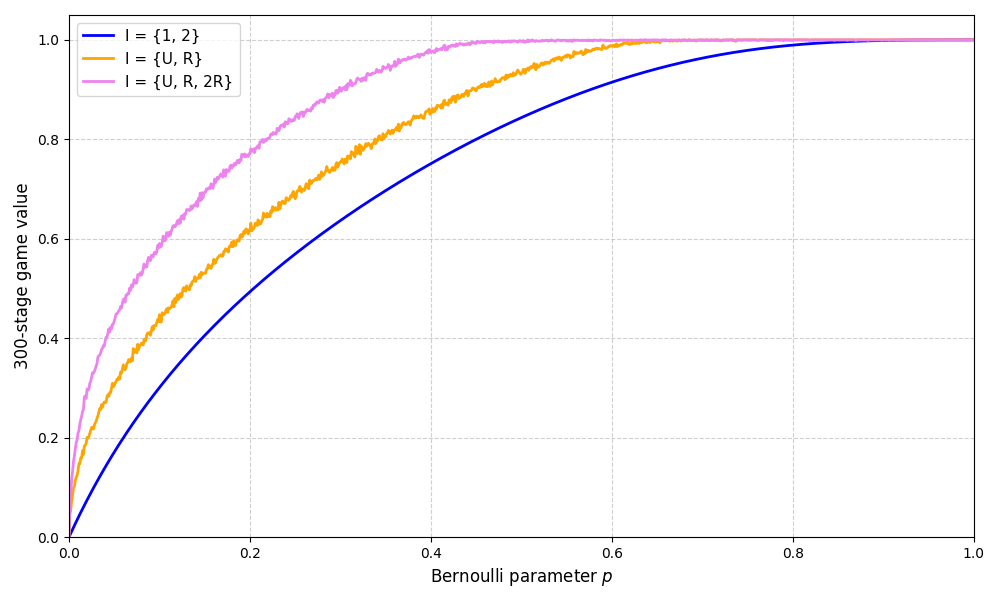}
        \caption{The graph presents simulations for Example \ref{ex:oriented-percolation-on-Z2} (orange) and Example \ref{ex:extended_moves} (violet).~These simulations were performed by playing the game up to 300 stages and 5 times for each $p$, where $p$ takes 1000 equally spaced values in [0, 1]. In blue the value function $p \mapsto v_{\infty}^{1,2}(p)$, given explicitly by \eqref{eq:limit_value_explicit}.}
        \label{fig:simulations2}
    \end{figure}
    
    In Figure \ref{fig:simulations2}, we compare the example related to oriented percolation with the following Example \ref{ex:extended_moves} which is identical except that it incorporates a dependent action. It appears that like Example \ref{ex:oriented-percolation-on-Z2}, Example \ref{ex:extended_moves} exhibits a phase transition at which the limit value becomes equal to $1$. We expect this transition to coincide with the phase transition of a corresponding percolation process and conjecture that this process is the one on the square lattice whose admissible paths have step set $\{\BFe_1, \BFe_2, 2\BFe_1\}$.

    \begin{example}\label{ex:extended_moves}
        A PMDP on $\bbZ^2$ with action set $I = \{\BFe_1, \BFe_2, 2\BFe_1\}$ and Bernoulli$(p)$ payoffs. 
    \end{example}

    In the same figure we also compare the second element (whose uniform value is determined explicitly in the next section) of the sequence $(v_{\infty}^{1,m})_{m \ge 1}$ with an approximate estimate of its limit $v_{\infty}^{(2)}$.~We conjecture that for each $m \ge 1$, $v_{\infty}^{1, m} \leq v_{\infty}^{1, m+1}$, although proving this inequality appears to be surprisingly difficult. 
    
\section{The case of dimension 1 with Bernoulli payoffs}\label{sec:dimension1}
    Consider three one-dimensional PMDPs with $\text{Bernoulli}(p)$ payoffs on the vertices (see Definition~\ref{def:def_onthevertices}) and respective action sets $\{1, 2\}$, $\{2, 3\}$ and $\{1, 3\}$. We denote these by $\Lambda_{1, 2}$, $\Lambda_{2, 3}$ and $\Lambda_{1, 3}$, respectively. We also consider the following edge-based variant of the former.

    \begin{example}\label{ex:dimension1}
        A PMDP on $\bbZ$ with $I = \{1, 2\}$ and  $\text{Bernoulli}(p)$ payoffs on the edges. 
    \end{example}

    In this subsection, we give stationary 0-optimal strategies for these four examples. For some of them, we establish the stronger property of optimality for every finite horizon. We then apply the dynamic programming principle (DPP) to derive a system of equations from which an expression for the limit value is deduced. The PMDP $\Lambda_{1,3}$ requires a different approach, developed in the appendix. For payoffs on the vertices, we also introduce a natural preorder on subsets of $\bbN_+$, from which derive several relations.

    These examples are interesting for two main reasons. First, for a singleton action set $I = \{a\} \neq \{\BFzero\}$, the SLLN yields $v_{\infty} = \E g$. These examples demonstrates that introducing an additional single action, even in the simplest one-dimensional setting, fundamentally alters the problem and characterizing 0-optimal strategies becomes nontrivial. Second, these examples indicate that the vertex-based model is simpler than its edge-based counterpart. In particular, in dimension 1 with action set $\{1, 2\}$, the simplest myopic optimal strategy (a strategy is said to be myopic if player's decision rule at each vertex depends only on a finite neighborhood) proposed in Proposition~\ref{prop:gamma_1} for payoffs on the vertices fails to be optimal when payoffs are on the edges.

    \subsection{Three examples with payoff on the vertices}\label{subsec:onthevertices}
    We say that a stationary strategy is $k$-myopic if it only depends on the realization of the payoffs at the next $k$ positions to the right of the current position. In that case we write, for example $\sigma(011) = +2$ to say that the (3-myopic) strategy $\sigma$ plays the action $+2$ whenever the 3 payoffs to the right of the current position are $0,1,1$ in that order. A letter is a placeholder for either a 0 or a 1. For example, $\sigma(aba) = +3$ means that the action is $+3$ whenever the first and third payoffs to the right of the current position are identical (and the second one is irrelevant). The strategy depends on $\omega$, but we omit it for the sake of readability.

    \begin{proposition}\label{prop:gamma_1} For $\Lambda_{1,2}$, the 1-myopic strategy
        \begin{equation*}
            \sigma^*(0) = + 2, \ \sigma^*(1) = + 1
        \end{equation*}
        is optimal for any $p$, and the uniform value is $v_\infty(p)=2p-p^2$.
    \end{proposition}

    \begin{proof}
        We prove that $\sigma^*$ is optimal in any finitely repeated PMDP, by proving that it weakly dominates (never provides a lower payoff than) any other strategy. Let $\sigma$ be any strategy and fix the initial state $z\in\bbZ$, an environment $\omega \in \Omega$, and a horizon $n \geq 1$. Let $g_1, \ldots, g_n$ denote the sequence of payoffs given by $\sigma$ in the $n$-stage decision process.

        First assume that $g_\omega(z+1)=1$ and that $\sigma(z)=+2$. Define $\sigma'$ as the strategy that mimics $\sigma$ except that $\sigma'(z)=+1$ and $\sigma'(z+1)=+1$. Then the strategy $\sigma'$ will give a sequence of payoffs $1,g_1,\ldots, g_{n-1}$, which implies that $\sigma'$ weakly dominates $\sigma$. Assume now that $g_\omega(z+1)=0$ and that $\sigma(z)=+1$. There are two subcases. If $\sigma(z+1)=+1$, let $\sigma'$ be identical to $\sigma$ except that $\sigma'(z)=+2$. Then the strategy $\sigma'$ will give a sequence of payoffs $g_2,\ldots, g_{n},u$ for some $u\in\{0,1\}$, which implies that $\sigma'$ weakly dominates $\sigma$. If $\sigma(z+1)=+2$, let $\sigma'$ be identical to $\sigma$ except that $\sigma'(z)=+2$ and  $\sigma'(z+2)=+1$. Then the strategy $\sigma'$ will give a sequence of payoffs $1, g_2,\ldots, g_{n}$, which implies that $\sigma'$ strictly dominates $\sigma$.

        Iterating this reasoning we see that  $\sigma^*$ weakly dominates any other strategy and is thus optimal. 
        
        By the DPP with the convention $v_0(z) = 0$ for all $z$, we get
        \[
            \E_p[nv_n] = p(1+\E_p[(n-1)v_{n-1}]) + (1-p)(p + \E_p[(n-1)v_{n-1}]) = n(2p - p^2),
        \]
        where we use the stationarity of $v_n(z)$ in $z$. Hence, $\E_p[v_{n}] = 2p - p^2$, and by the Dominated Convergence Theorem, $v_{\infty}(p) = 2p - p^2$.  \end{proof}

    \begin{proposition}\label{prop:gamma_2} For $\Lambda_{2,3}$, the 4-myopic strategy
        \begin{equation*}
            \sigma^*(a01b) = +3, \sigma^*(a10b) = +2, \sigma^*(abb0) = +3, \sigma^*(abb1) = +2
        \end{equation*}
        is optimal for any $p$, and the uniform value is 
        \[
            v_\infty(p) = \frac{p(3 - 3p + p^3)}{1 + p - 3p^2 + 2p^3}.
        \]
    \end{proposition}

    \begin{proof}
        The proof follows the same line of argument as the proof of Proposition~\ref{prop:gamma_1}. Let $\sigma$ be any strategy and fix $z\in\bbZ$, $\omega\in\Omega$, and $n \geq 1$. For each of the 4 possible differences with $\sigma^*$ we will give a strategy $\sigma'$ that differs from $\sigma$ only at the beginning and weakly dominates it; we let the reader check that the sequence of payoffs is indeed at least as good for $\sigma'$ than for $\sigma$.

        First, assume the future is $a01b$ and $\sigma(z)=+2$. If $\sigma(z+2)=+3$, then doing ``+3 then +2'' strongly dominates $\sigma$. If $\sigma(z+2)=+2$ and $\sigma(z+4)=+3$ then doing ``+3 then +2 then +2'' weakly dominates $\sigma$. If $\sigma(z+2)=+2$ and $\sigma(z+4)=+2$ then doing ``+3 then +3'' weakly dominates $\sigma$. 
        
        Second, assume the future is  $a10b$ and $\sigma(z)=+3$.  If $\sigma(z+3)=+2$, then doing ``+2 then +3'' strongly dominates $\sigma$. If $\sigma(z+3)=+3$, then doing ``+2 then +2 then +2'' weakly dominates $\sigma$. 
        
        Third, assume the future is $abb0$ and $\sigma(z)=+2$. If $\sigma(z+2)=+3$, then doing ``+3 then +2'' gives the same sequence of payoffs as $\sigma$. If $\sigma(z+2)=+2$ and $\sigma(z+4)=+3$ then doing ``+3 then +2 then +2'' weakly dominates $\sigma$. If $\sigma(z+2)=+2$ and $\sigma(z+4)=+2$ then doing ``+3 then +3'' weakly dominates $\sigma$. 
        
        Finally, assume the future is  $abb1$ and $\sigma(z)=+3$. If $\sigma(z+3)=+2$, then doing ``+2 then +3'' gives the same sequence of payoffs as $\sigma$. If $\sigma(z+3)=+3$, then doing ``+2 then +2 then +2'' weakly dominates $\sigma$.

        Iterating this reasoning we see that $\sigma^*$ weakly dominates any other strategy and is thus optimal. 

        Denote now $w_n=\mathbb{E}[nv_n(z)]$, and $w_n^1=\mathbb{E}[nv_n(z) \mid g_\omega(z+2)=1]$ the conditional expectation of the value of the $n$-stage decision process, given that the payoff in the state 2 positions to the right is 1. Since $\sigma^*$ is optimal, the DPP implies that
        % \begin{align*}
        %     & w_n = 2p(1-p)(1+w_{n-1}) + p^2(1-p)(1+w_{n-1}) \\ 
        %     & \qquad + (1-p)^3w_{n-1} + p^3(1 + w^1_{n-1}) + p(1-p)^2w^1_{n-1}.
        % \end{align*}
        % Moreover, when $g_\omega(z+2)=1$, $\sigma^*$ gives
        % \[
        %     w^1_n=(1-p)(1+w_{n-1})+p(1-p)(1+w_{n-1})+p^2(1+w^1_{n-1}).
        % \]
        % Letting $w^1_n-w_n=w^1_n-w_n$ yields
        \begin{eqnarray}
            w_n&=&2p-p^2+w_{n-1}+p(2p^2-2p+1)(w^1_{n-1}-w_{n-1})\\
            \label{2eequaVdelta}
            w^1_n&=&1+w_{n-1}+p^2(w^1_{n-1}-w_{n-1}).
        \end{eqnarray}
        Subtracting the first equation from the second we get 
        \[
            w^1_n-w_n=(1-p)^2-p(1-p)(1-2p)(w^1_{n-1}-w_{n-1}).
        \]
        Since $|p(1-p)(1-2p)|<1$ for any $p\in[0,1]$, 
        \[
           w^1_n-w_n \rightarrow \frac{(1-p)^2}{1+p-3p^2+2p^3}, \quad \text{as } n \to \infty.
        \]
        From \eqref{2eequaVdelta}, $w_n\sim n (1-(1-p^2)\delta)$ which yields the result.  
    \end{proof}

    \begin{proposition} \label{propex13}
        Let $p^*\sim 0.26102$ be the only real root of $2X^3-5X^2+5X-1$. For $p\in[0,p^*]$, the 2-myopic strategy
        \begin{equation*}
            \sigma^*(00) = +3, \sigma^*(01) = +1, \sigma^*(10) = +1, \sigma^*(11) = +1
        \end{equation*}
        is optimal for $\Lambda_{1,3}$. For $p \in (p^*,1)$, no myopic strategy is optimal; an optimal strategy is given by
        \begin{eqnarray*}
            \sigma^*(00)&=&+3\\
            \sigma^*(010^k1)&=&+3 \text{ if } k \equiv 0\Mod{3} \text{ or } k \equiv 1\Mod{3}\\
            \sigma^*(010^k1)&=&+1 \text{ if } k \equiv 2\Mod{3}\\
            \sigma^*(10^k1)&=&+1 \text{ if } k \equiv 0\Mod{3} \text{ or } k \equiv 2\Mod{3}\\
            \sigma^*(10^k1)&=&+3 \text{ if } k \equiv 1\Mod{3}.
        \end{eqnarray*}
        The uniform value is 
        \[
            v_\infty(p) = \begin{cases}
                            p(3-3p+p^{2})/(1+p-p^2) & \text{ if } p\leq p^*,\\
                            \frac{p\left(28 - 90p + 127p^2 - 98p^3 + 43p^4 - 10p^5 + p^6\right)}{12 - 29p + 25p^2 - 2p^3 - 10p^4 + 6p^5 - p^6} & \text{ if } p\geq p^*.
                        \end{cases}
        \]
    \end{proposition}

    The proof of this example is considerably more technical than the previous ones and for the sake of readability, we relegate it to Appendix \ref{sec:appendix}.

    \begin{remark} While for this example there are no myopic strategies in some range of $p$, for every $p$ there exists a finite automaton that plays 0-optimally. It is clear for $p\leq p^*$, and above $p^*$ the optimal strategy amounts to counting the length modulo 3 of a sequence of zeros, which can be implemented by an automaton.
    \end{remark}

    \subsection{A natural preorder on nonempty finite subsets of $\bbN_+$} Let $\mathcal{P}_f(\bbN_+)$ be the set of nonempty finite subsets of $\bbN_+$. For any $I\in\mathcal{P}_f(\bbN_+)$ and $p\in[0,1]$ we let $v_\infty^I(p)$ denote the uniform value of the decision process with action set $I$ and Bernoulli$(p)$ payoffs on the vertices. 
    
    \begin{definition}[Preorder on $\mathcal{P}_f(\mathbb{N}_+)$] $I\preceq J$ if and only if $v_\infty^I(p)\leq v_\infty^J(p)$ for every $p\in[0,1]$.
    \end{definition}
    
    Clearly this preorder is finer than the inclusion: $I\subset J \implies  I\preceq J$. Also, this is not an order: for example, $I\preceq J$ and $J\preceq I$ whenever there exists $a$ and $b$ in $\mathbb{N}_+$ such that $aI = bJ$. From previous propositions, we have $\{1,2\}\preceq \{1,3\}\preceq \{2,3\}$. In fact, we can say more.

    \begin{proposition} It holds that
        \begin{itemize}[noitemsep,topsep=0pt]
            \item $\{1,2\}\preceq I$ for any $I$ with at least two elements, and
            \item $\{2,3\}\preceq I$ for any $I$ with at least two elements that do not divide each other.
        \end{itemize}
            
    \end{proposition}

    \begin{proof}
        For the first part, let $a<b$ be two distinct elements of an action set $I$. Consider the following stationary strategy that at state $z$. The strategy plays $a$  when the payoff in state $z+a$ is 1. It plays $b$ when the payoff in state $z+a$ is 0. Computing the payoff of this strategy gives exactly the same equations as in Proposition~\ref{prop:gamma_1}. Since the player can at least guarantee the payoff of this strategy, the statement follows.

        For the second part, let $a<b$ be two distinct elements of an action set $I$ and assume $a$ does not divide $b$. Consider the stationary strategy that in state $z$ plays $a$ either: when the payoff in state $z+a$ is 1 and the payoff in state $z+b$ is 0; or when the payoff in state $z+a$ is equal to the payoff in state $z+b$, and the payoff in state $z+2a$ is 1, and plays $b$ either; when the payoff in state $z+a$ is 0 and the payoff in state $z+b$ is 1; or when the payoff in state $z+a$  is equal to the payoff  in state $z+b$, and the payoff in state $z+2a$ is 0.

        Computing the payoff of this strategy gives exactly the same equations as in Proposition~\ref{prop:gamma_2} with the crucial point being that state $z+b$ is inaccessible from state $z+a$, because of the assumption that $a$ does not divide $b$. Since the player can at least guarantee the payoff of this strategy, the second statement follows.   
    \end{proof}

    It is not immediately obvious whether this preorder is total. We now explain why it is not.

    \begin{proposition}
        For $m \gg 1$, $I_m:=\{1,2,3,\ldots,m\}$ and $J_m:=\{1,m+1\}$ are not comparable.
    \end{proposition}

    \begin{proof}
        By the same arguments as in the proof of Proposition~\ref{prop:gamma_1}, for any $m$ the optimal strategy for action set $I_m$ is to play the least action that gives a payoff of 1 if there is any, and play $m$ if the next $m$ payoffs are all 0. This implies that $v_\infty^{I_m}(p)=1-(1-p)^m$. In particular for a fixed $m$, $v_\infty^{I_m}(p)\sim mp$ as $p$ goes to 0; and for a fixed $p>0$, $v_\infty^{I_m}(p) \rightarrow 1$ as $m$ goes to $\infty$.

        For $J_m$ we cannot do a precise computation of the value  $v_\infty^{J_m}(p)$. However, for a fixed $p$, by the dimensional lifting property (see Theorem~\ref{thm:convergence_game_m} in Section \ref{sec:dimensional-lifting-oriented-percolation}), we know that $v_\infty^{J_m}(p)$ tends to $v_{\infty}^{(2)}(p)$ as $m \to \infty$. Recall that by Proposition~\ref{prop:bernoulli-continuity}, $v_{\infty}^{(2)}(p)$ is continuous at $p = 0$. This implies that for $p$ fixed and small enough, $v_\infty^{J_m}(p)$ does not tend to 1 as $m$ goes to infinity, and thus that $I_m \npreceq J_m$.

        On the other hand, fix $m$ and consider the following greedy strategy: play $1$ if any of the next $m$ payoffs is 1, and otherwise play $m+1$. Then, if the first payoff of 1 is $n$ steps to the right, it takes $\left\lfloor n/(m+1)\right\rfloor$ moves of $(m+1)$ and $n-(m+1)\left\lfloor n/(m+1)\right\rfloor$ moves of 1 to get there. This implies that under this strategy, the expected number of moves until the next payoff of 1 is
        \begin{equation*}
            \sum_{n=1}^{+\infty} p(1-p)^{n-1} \left(n-m\left\lfloor\frac{n}{m+1}\right\rfloor\right) < \sum_{n=1}^{+\infty} p(1-p)^{n-1} \left(\frac{n}{m+1}+m\right) = m+\frac{1}{(m+1)p}
        \end{equation*}
        Hence, $v_\infty^{J_m}(p)>1/\left(m+\frac{1}{(m+1)p}\right)\sim (m+1)p$ as $p$ goes to 0. Thus, for a fixed $m$, when $p$ is small enough $v_\infty^{J_m}(p)>v_\infty^{I_m}(p)$ and thus $J_m\npreceq I_m$ for every $m$.  
    \end{proof}

    \subsection{An example with payoff on the edges.}\label{subsec:ontheedges}
    For Example \ref{ex:dimension1} we follow the same approach. We use $\underline{a}$ to denote payoffs on the +1 edges, and $\overline{b}$ to denote payoffs on the +2 edges; for a block, we use $\underline{a}\overline{b}\underline{c}\overline{d}\ldots$, meaning that, if it starts at position $z$, then $g_{\omega}(z, +1) = a, g_{\omega}(z, +2) = b, g_{\omega}(z + 1, +1) = c, g_{\omega}(z + 1, +2) = d$, etc.

    \begin{proposition}\label{prop:dimension1-game}
    For Example \ref{ex:dimension1}, the strategy
    \begin{eqnarray*}
        \sigma^*(\underline{1}\overline{0}) & = & +1, \\
        \sigma^*(\underline{0}\overline{1}) & = & +2, \\
        \sigma^*(\underline{a}\overline{a}\underline{1}) & = & +1, \\
        \sigma^*(\underline{a}\overline{a}\underline{0}(\overline{1}\underline{0})^k\overline{c}\underline{d}) & = & \begin{cases}   
                        +1 & \text{if } k \text{ is odd and } (c,d)\neq (1,0), \\
                        +2 & \text{if } k \text{ is even and } (c,d)\neq (1,0).
                    \end{cases} 
    \end{eqnarray*}			
    is optimal for any finite horizon. The uniform value is given by
    \begin{equation}\label{eq:limit_value_explicit}
        v_{\infty}(p) = \frac{p^6 - 2p^5 - 2p^4 + 6p^3 - 4p^2}{2p^5 - 7p^4 + 7p^3 - 2p^2 - p}.
    \end{equation}
    \end{proposition}
        
    \begin{proof} Fix $\omega \in \Omega$ and let $z\in\bbZ$ be the starting state. Let $n \geq 1$, $x := g_{\omega}(z, +1)$ and $y := g_{\omega}(z, +2)$. 

    Assume first that $x = 1$ and $y = 0$ and let $\tau$ be a strategy that first plays action $+2$ and then optimally in the game with horizon $n-1$. Consider the strategy $\sigma$ that plays $+1, +1$ and mimics $\tau$ from state $z+2$ onward. Denote $g_t$ the payoff under $\tau$ at stage $t$. Then, since $g_1 = y = 0$
    \begin{equation*}
        \gamma_n^{\omega}(z, \sigma) = \frac{1}{n}\left(1+g_{\omega}(z+1,1) +\sum_{m=2}^{n-1} g_m\right) \geq \frac{1}{n}\left(\sum_{m=1}^{n} g_m\right).
    \end{equation*}				
    
    Hence $\sigma$ is at least as good as $\tau$. It follows that action $+1$ is optimal at state $z$ in that case.
    
    Second, if $x=0$ and $y=1$,  let $\tau$ be a strategy that first plays action $+1$ and then optimally in the game with horizon $n-1$. Assume first that it plays $+2$ in state $z+1$. Define the strategy $\sigma$ that plays $+2$ in state $z$, $+1$ in state $z+2$, and mimics $\tau$ from state $z+3$ on. The only difference in payoff between the two strategies occurs in the first two stages. During these two stages $\sigma$ will gain at least 1 (since $y=1$), while $\tau$ will gain at most 1 (since $x=0$). Hence, $\sigma$ is at least as good as $\tau$. On the other hand, assume that $\tau$ plays $+1$ in state $z+1$. Then we define $\sigma$ as the strategy that plays $+2$ in state $z$, mimics $\tau$ from state $z+2$ on, and plays arbitrarily in the last stage. Denote $g_t$ the payoff under $\tau$ at stage $t$. Since all payoffs are nonnegative, $g_1 = x = 0$ and $g_2 \leq 1$
    \begin{equation*}\gamma_n^{\omega}(z, \sigma) \ge \frac{1}{n}\left(1 +\sum_{m=3}^{n} g_m\right) \ge \frac{1}{n}\left(\sum_{m=1}^{n} g_m\right),
    \end{equation*}	
    and we see that $\sigma$ is at least as good as $\tau$. It follows that action $+2$ is optimal at state $z$ in that case.
    
    We assume from now on that $x=y=a$, and let $b := g_{\omega}(z+1, +1)$. If $b = 1$, as in the first case one easily sees that $+1+1$ is at least as good as $+2$ and thus that $+1$ is optimal at state $z$.
    
    Assume from now on that $b = 0$. As in the second case one easily sees that starting with $+2$ is at least as good as starting with $+1+1$. If in addition $g_{\omega}(z+1, +2)\leq g_{\omega}(z+2, +1)$, then starting with $+2+1$ is at least as good as starting with $+1+2$, and it follows that action $+2$ is optimal at state $z$. 
    
    Assume from now on that $g_{\omega}(z+1, +2)> g_{\omega}(2, +1)$, hence that $g_{\omega}(z+1, +2)=1$ and $g_{\omega}(z+2, +1)=0$. Starting by $+2+1$ thus gives a payoff of $a$ and $0$ in the first two stages, strictly worse than playing $+1+2$ that gives a payoff of $a$ and $1$. Assume in addition that $g_{\omega}(z+2, +2)\leq g_{\omega}(z+3, +1)$. Let $\tau$ be any strategy that plays $+2+2$ and then optimally, and define $\sigma$ the strategy that plays $+1+2+1$ and then mimics $\tau$. Denote $g_t$ the payoff under $\tau$ at stage $t$. Then, since $g_{\omega}(z+2, +2)\leq g_{\omega}(z+3, +1)$
    \begin{equation*}\gamma_n^{\omega}(z, \sigma) = \frac{1}{n}\left(a+1+g_{\omega}(z+3,1) +\sum_{m=3}^{n-1} g_m\right) \ge \frac{1}{n}\left(\sum_{m=1}^{n} g_m\right).
    \end{equation*}
    and we see that $\sigma$ is at least as good as $\tau$. It follows that action $+1$ is optimal at state $z$ in that case.
    
    Finally, assume that  $g_{\omega}(z+2, +2)>g_{\omega}(z+3, +1)$, hence that $g_{\omega}(z+2, +2)=1$ and $g_{\omega}(z+3, +1)=0$. Then after a first move of either $+1$ or $+2$ it is optimal (by the second case above) to play a second move of $+2$ (with a payoff of 1). One can then proceed by induction by removing states $z+3$ and $z+4$ and considering the game with horizon $n-1$, see Figure \ref{fig:kexceedsn}.
    \begin{figure}[!ht]
    \centering
    \begin{tikzpicture}[node distance=1.5cm]
        % Nodes
        \foreach \x in {0,1,2,3,4,5,6} {
            \node[circle, fill=black, inner sep=1.5pt] (n\x) at (\x*1.2, 0) {};
            % \node[below, 0.5pt] at (n\x) {$z+\x$};
        }
        \node at (0, 0) {$z$};
        \node at (7.6, 0) {$\dots$};

        \draw[->] (n0) to (n1);

        % Zigzag line
        \draw[thick, decoration={zigzag, segment length=4pt, amplitude=1pt}, decorate] (n1) -- (n6);

        % Arcs above
        \draw[->, red] (n0) to[bend left=25] node[above, blue] {$a$} (n2);
        \draw[->, red] (n2) to[bend left=25] node[above, blue] {$1$} (n4);
        \draw[->, red] (n4) to[bend left=25] node[above, blue] {$1$} (n6);
        
        % Arcs below
        \draw[->] (n1) to[bend right=25] node[below, blue] {$1$} (n3);
        \draw[->] (n3) to[bend right=25] node[below, blue] {$1$} (n5);

        \node[blue] at (0.6, -0.2) {$a$};
    \end{tikzpicture}
    \caption{Edges with payoff 0 are drawn as zigzags; all other payoffs are shown in blue.}
    \label{fig:kexceedsn}
    \end{figure}

    We now compute the expected payoff of the optimal strategy $\sigma^*$, that is $\lim_{n\to\infty}\E \gamma_n(z, \sigma^*) $. By Hardy-Littlewood Tauberian theorem \citep[Theorem 2.1]{Sznajder1992}, this limit equals $\lim_{\lambda\to 0}\E \gamma_\lambda(z, \sigma^*) $ where  $\gamma_\lambda$ represents the (normalized) discounted payoff
    \[
        \gamma_\lambda^{\omega}(z, \sigma) := \lambda \sum_{m=1}^{+\infty}(1-\lambda)^{m-1}g_{\omega}(z_m, \sigma_m(\omega, z))  .  
    \]
    
    Denote $w_\lambda := \E \gamma_\lambda(z, \sigma^*) $, which is independent on $z$, and $w_{\lambda}^x := \E[v_{\lambda} \mid g(z, +1) = x]$.
    
    For $k \geq 0$, let $q := p(1 - p)$, $a_k := q^k(1 - q)$, $K(k) := 1 + \left \lfloor (k+1)/2 \right \rfloor$
    \[
        G_{K, j}(\lambda) := \lambda\sum_{t=j}^{K}(1-\lambda)^{t-1} \, \text{for } j \in \{1, 2\} \quad \text{and} \quad \ F_K(w_{\lambda}^0, w_{\lambda}^1) := \frac{p w_{\lambda}^1 + (1 - p)^2 w_{\lambda}^0}{1 - q}(1 - \lambda)^{K}.
    \]
    Since $\sigma^*$ is optimal, by the DPP the following system of equations is deduced 
    \begin{eqnarray*}
        w_{\lambda} & = & p w_{\lambda}^1 + (1-p) w_{\lambda}^0, \\
        w_{\lambda}^1 & = & (1 - p)(\lambda + (1 - \lambda) w_{\lambda}) + p^2 (\lambda + (1 - \lambda) w_{\lambda}^1) + q \sum_{k = 0}^{\infty} a_k \left( G_{K(k), 1} + F_{K(k)}(w_{\lambda}^0, w_{\lambda}^1)\right), \\
        w_{\lambda}^0 & = & p(\lambda + (1-\lambda) w_{\lambda}) + q (1-\lambda) w_{\lambda}^1 + (1 - p)^2 \sum_{k=0}^{\infty} a_k \left(G_{K(k), 2} + F_{K(k)}(w_{\lambda}^0, w_{\lambda}^1)\right).
    \end{eqnarray*}
    Solving the three-variable system and then letting $\lambda$ go to 0 yields the expression for $\lim_{\lambda \to 0} w_{\lambda}$. By the Dominated Convergence Theorem
    , \eqref{eq:limit_value_explicit} follows.  
    \end{proof}

    \begin{remark}\label{rem:dichotomy} 
        For this example, the binary nature of the payoffs permits a direct decision rule when immediate payoffs do not tie. By contrast, when payoffs follow a continuous distribution, ties occur with probability zero, so the immediate comparison $g(z, +1) \neq g(z, +2)$ almost surely determines the locally preferable action. Nevertheless, selecting the action with the larger immediate payoff need not be optimal globally. Even when $y := g(z, +2) > g(z, +1) =: x$, the difference $y - x$ may be arbitrarily small, while the continuation values may satisfy $v_{n}(z+1) \gg v_{n}(z+2)$.
    \end{remark}
    
\section{Connection with directed last-passage percolation}\label{sec:lpp}
    Consider a lattice $(\bbZ^d, \mathbb{L}^d)$ where the \emph{constant path length property} is assumed, meaning that, for any $\BFx,\BFy \in \bbZ^d$, any path with edges in $\mathbb{L}^d$ connecting $\BFx$ and $\BFy$ always has the same length. Let us consider a collection $\{c(e)\}_{e \in \mathbb{L}^d}$ of i.i.d.\ uniformly bounded random weights. For a path $\chi$ in $\bbZ^d$ with edges in $\mathbb{L}^d$, let $|\chi|$ denote the number of edges in $\chi$, and let $\chi_0$ denote its starting point. Consider the \emph{passage time} defined by
    \[
        l(\chi) = \sum_{e \in \chi} c(e).
    \]
    
    \emph{Directed last-passage percolation} (LPP) (see \citep{DamronRassoulASepp2016} for a concise introduction or \citep[Section~7.2.3]{AuffingerDamronHanson2017}) concerns the \emph{last-passage time} (maximum accumulated weight) along directed paths in a random medium. Here we first present two existing formulations of LPP: the standard one, in which directed paths maximize their accumulated weight between two fixed points, and an alternative one, in which paths start from a fixed point and maximize their accumulated weight over endpoints lying on a given hyperplane. For oriented PMDPs, the first formulation provides insight into the maximizing direction of the play when the set of actions is independent. The second is a particular case of oriented PMDPs. We also briefly review the available results for LPP models, from which we can infer corresponding results for PMDPs. Finally, in Section \ref{sec:variable-length-quantity}, we present another fixed-endpoint quantity along a characterization of the uniform value of oriented PMDPs in terms of it.

    \textbf{LPP model with a fixed endpoint.}~In standard last-passage percolation, the main quantity of interest is the last-passage time defined by
    \begin{equation}\label{eq:fixed-endpoint-lpt}
        G_{\BFzero, \lfloor \BFx\rfloor} := \max_{\chi \colon \BFzero \to \lfloor \BFx\rfloor} l(\chi), \quad \text{for every } \BFx \in \bbR_+^d,
    \end{equation}
    where the maximum is taken over all paths from $\BFzero$ to $\lfloor \BFx\rfloor$ with edges in $\mathbb{L}^d$, and $\lfloor \BFx\rfloor$ is taken coordinatewise.
    
    Under standard assumptions on the weights, satisfied in particular by i.i.d.\ uniformly bounded weights, there exists a deterministic \emph{shape function} $\mu \colon \bbR_{+}^d \to \bbR$ such that, for every $\BFx \in \bbR_+^d$, it holds almost surely that \citep[Remark 2.2 and Theorem 2.4(b)]{GeorgiouFirasSeppalainen2016}
    \begin{equation}\label{eq:convergence_shape_function-1}
        \lim_{n\to\infty} n^{-1}G_{\BFzero,\lfloor n\BFx\rfloor} = \mu(\BFx).
    \end{equation}
    It is also proven that $\mu$ is superadditive, homogeneous, symmetric, and concave. Superadditivity and homogeneity follow quite naturally from \eqref{eq:fixed-endpoint-lpt} and \eqref{eq:convergence_shape_function-1}, while symmetry results follows from the symmetry of distribution. Then, superadditivity and homogeneity together imply concavity.
    
    For oriented PMDPs with a linearly independent set of actions, the following relation holds. In particular, it is saying that for the canonical action set the diagonal is a maximizing direction. Equivalently, restricting the endpoint to $n\BFone_d$ causes no loss in the limit value.
    
    \begin{proposition}\label{prop:p2p-uniform-value}
        Consider an oriented PMDP in $\bbZ^d$ with a set $I$ of $d$ linearly independent actions. Then
        \[
            v_{\infty} = \frac{\mu(\BFone_{d})}{d}.
        \]
    \end{proposition}

    \begin{proof} Identify $I$ with the canonical action set $I_C := \{\BFe_1, \ldots, \BFe_d\}$. The latter is a particular case where each path from $\BFx$ to $\BFy$ has the same length $\|\BFy - \BFx\|_1$.
    
    Define $w_n := dn v_{dn}(\BFzero) = \max\{G_{\BFzero,\BFy} \colon \BFy \in \mathbb Z_+^d, \|\BFy\|_1 = dn\}$. We have
        \begin{equation}\label{eq:direct-one}
            w_n \ge G_{\BFzero, n \BFone_d} \implies \liminf_{n \to \infty}\frac{w_n}{n} \ge \mu(\BFone_d).
        \end{equation}
        
        Now, let $C := \|c(e)\|_\infty$. If $\BFx \le \BFy$ coordinatewise, extending an optimal path to $\BFy$ from any directed path to $\BFx$, and since each additional weight is at least $-C$, we obtain the one-side bound
        \begin{equation}\label{eq:one-side-bound}
            G_{\BFzero, \BFx} - C\|\BFy - \BFx\|_1 \le G_{\BFzero, \BFy}.
        \end{equation}
        Fix an integer $s \ge 1$. For each $\BFy = \arg\max_{\BFy \in \mathbb Z_+^d, \|\BFy\|_1 = dn} G_{\BFzero, \BFy}$, set
        \begin{equation*}
            Q_s := \left\{\BFq \in \mathbb Z_+^d: \BFy \le \lfloor n\BFq \rfloor \text{ and } \|\lfloor n\BFq\rfloor - \BFy\|_1 \le \frac{dn}{s}\right\}.
        \end{equation*}
        Note that 
        \begin{equation}\label{eq:Q_s-def}
            Q_s \subset \left\{\BFq \in s^{-1}\mathbb Z_+^d: d\le\|\BFq\|_1\le d+\frac ds \right\}.
        \end{equation}
        Using \eqref{eq:one-side-bound} and because $Q_s$ is finite
        \[
            \frac{w_n}{n} \le \max_{\BFq\in Q_s}\frac{G_{\BFzero,\lfloor n\BFq\rfloor}}n + \frac{Cd}{s}.
        \]
        Taking $\limsup_{n \to \infty}$, the pointwise almost-sure convergence in \eqref{eq:convergence_shape_function-1} yields
        \[
            \limsup_{n\to\infty}\frac{w_n}{n} \le \max_{\BFq\in Q_s}\mu(\BFq)+\frac{Cd}{s}.
        \]
        Next, concavity, symmetry, and homogeneity of $\mu$ with \eqref{eq:Q_s-def} imply
        \[
            \mu(\BFq) \le \mu\!\left(\frac{\|\BFq\|_1}{d}\BFone_d\right) = \frac{\|\BFq\|_1}{d}\mu(\BFone_d) \le \mu(\BFone_d)+\frac{|\mu(\BFone_d)|}{s}.
        \]
        Hence
        \[
            \limsup_{n \to \infty}\frac{w_n}{n} \le \mu(\BFone_d) +\frac{Cd + |\mu(\BFone_d)|}{s} \xrightarrow[s \to \infty]{} \mu(\BFone_d).
        \]
        Combining it with \eqref{eq:direct-one} we obtain that almost surely
        \[
            \frac{w_n}{n} \xrightarrow[n \to \infty]{} \mu(\BFone_d) \implies v_{dn}(\BFzero) \xrightarrow[n \to \infty]{} \frac{\mu(\BFone_d)}d,
        \]
        which proves the desired identity by invoking Theorem~\ref{thm:iid_and_oriented_percolation_games}.   
    \end{proof}

    \textbf{LPP model with a fixed number of steps.}~Another variant of LPP fixes the number of steps rather than the spatial endpoint. The last-passage time in this case is given by \citep[Remark 2.2 and Equation (2.4)]{GeorgiouFirasSeppalainen2016}
    \begin{equation}\label{eq:fixed-step-lpt}
        G_{\BFzero, (n)} := \max_{\substack{\chi \colon |\chi| = n \\ \chi_0 = \BFzero}} l(\chi), \quad \text{for every } n \geq 1,
    \end{equation}
    where the maximum is taken over all paths with edges in $\mathbb{L}^d$ starting at $\BFzero$ and having length $n$. A law of large numbers also holds for this quantity: there exists $\mu_{\text{pl}} \in \bbR$ such that \citep[Theorem~2.4(a)]{GeorgiouFirasSeppalainen2016}
    \[
        \lim_{n \to \infty} n^{-1}G_{\BFzero, (n)} = \mu_{\text{pl}}, \quad \text{almost surely}.
    \]

    \begin{proposition}\label{prop:p2l-uniformvalue}~Consider an oriented PMDP in $\bbZ^d$. Maximize the expression in \eqref{eq:fixed-step-lpt} over paths using edges in the edge set of the transition graph of the PMDP instead of on $\mathbb{L}^d$. Then, 
        \begin{equation}\label{eq:p2l-uniformvalue}
            \mu_{\mathrm{pl}} = v_{\infty}.
        \end{equation}
    \end{proposition}
    \begin{proof}
        If the maximum in \eqref{eq:fixed-step-lpt} is taken over paths with edges in the edge set of the transition graph of the PMDP, then $G_{\BFzero, (n)} = n v_n(\BFzero)$ for every $n \geq 1$. Invoking Theorem~\ref{thm:iid_and_oriented_percolation_games} yields the claim.
     \end{proof}

    % \begin{remark}[Duality between the two formulations and related models]\label{rem:literatureLPP}
    %     For the duality between the two formulations of LPP presented, see \citep[Section~4]{GeorgiouFirasSeppalainen2016}. These models are also closely related to \emph{directed polymers in a random environment} \citet{Comets2017, Zygouras2024}; in particular, they arise as zero-temperature limits of directed polymer models \citep[Section~2]{GeorgiouFirasSeppalainen2016}. Another related model is \emph{greedy lattice paths} \citet{GandolfiKesten1994}, an undirected version of LPP with weights on the vertices with a fixed number of steps and self-avoiding paths.
    % \end{remark}

    \begin{remark}\label{rem:resultLPP} 
        Variational characterizations of the LPP shape functions $\mu$ and $\mu_{\mathrm{pl}}$ are given in \citep[Theorems~3.2, 4.4, 7.2 and 7.3]{GeorgiouFirasSeppalainen2016}. Combining Propositions~\ref{prop:p2p-uniform-value} and \ref{prop:p2l-uniformvalue}, one obtains variational formulas for oriented PMDPs with actions sets satisfying the constant path length property. Because these variational formulas are defined on infinite-dimensional spaces of functions or measures, extracting explicit information from them is difficult. In two dimensions, however, certain directed models are \emph{exactly solvable}. This in particular means that their shape functions can be evaluated in closed form. For standard fixed-endpoint LPP with exponential weights of mean $1$ on the vertices instead of the edges, the shape function is given by $\mu(x_1, x_2) = (\sqrt{x_1} + \sqrt{x_2})^2$, $\text{for every } (x_1, x_2) \in \mathbb{Q}_+^2$ (first shown by  \citet{Rost1981}). For geometric weights with parameter $p$ and positive integer convention, \citet{Johansson2000} proved that for every $(x_1, x_2) \in \mathbb{Q}^2_+$ the shape function is given by
        \[
            \mu(x_1, x_2) = \frac{1}{p}\left(x_1 + x_2 + 2\sqrt{x_1 x_2 (1-p)}\right).
        \]
        If one were to relax the uniform boundedness assumption in our PMDP, then the limit value of Example \ref{ex:general-oriented-percolation} with $F$ the exponential distribution of mean $1$ and payoffs on the vertices, would be $2$. Moreover, in these cases, it is known that the fluctuations of the last-passage time $G_{\BFzero, \lfloor n(x_1, x_2)\rfloor}$ with $x_1, x_2 > 0$ around $n\mu(x_1, x_2)$ are of order $n^{1/3}$, the transversal deviation of the optimal path is of order $n^{2/3}$, and the centered and scaled last-passage times converge to a non-degenerate limit distribution \citep[see][Section 13.1]{longestsubsequence2022}. This behavior place them in the KPZ universality class \citep[see][]{Corwin2016}.
    \end{remark}

    \subsection{A related variable-length quantity}\label{sec:variable-length-quantity}
    In the first formulation, the path length is fixed automatically because every path between two given points has the same number of steps. For oriented PMDPs this need not be true (for instance, Example \ref{ex:extended_moves}). We now move beyond the constant path length property and introduce another related quantity that yields an alternative characterization of the uniform value of oriented PMDPs.

    Let $\mathcal{R}$ denote the set of points reachable from the origin in any number of stages and using actions in $I$. For $\BFx \in \mathcal{R}\setminus\{\BFzero\}$, define
    \begin{equation*}
        \mathcal{L}_{\BFzero, \BFx} = \max_{\substack{\chi \colon \BFzero \underset{I}{\to} \BFx}} \frac{l(\chi)}{|\chi|},
    \end{equation*}
    where the maximum is taken over paths from $\BFzero$ to $\BFx$ using actions in $I$. Since $I$ is finite and oriented, only finitely many paths join $\BFzero$ to a given reachable point $\BFx$, so the maximum is attained.
    
    \begin{proposition}\label{prop:characterization-limit-value-with-L}
        Consider an oriented PMDP with action set $I$. Then,
        \[
            v_{\infty} = \lim_{n \to \infty} \max_{\BFx \in \mathcal{R}_n} \mathcal{L}_{\BFzero, \BFx}, \quad \Pbb\text{-a.s.}
        \]
        where $\mathcal{R}_n$ is the set of points reachable from the origin in exactly $n$ steps using actions in $I$ .
    \end{proposition}

    \begin{proof}
        Define $G_n = \max_{\BFx \in \mathcal{R}_n} \mathcal{L}_{\BFzero, \BFx}$. For $n \geq 1$, an optimal path in the $n$-stage game gives 
        \[
            v_n(\BFzero) \le G_n.
        \]
        From orientability there exists $\BFu \in \bbR^d \setminus \{\BFzero\}$ such that $\BFa \cdot \BFu > 0$ for all $\BFa \in I$. Define 
        \[
            \delta = \min_{\BFa \in I} \BFu \cdot \BFa \quad \text{and} \quad \Delta = \max_{\BFa \in I}\BFu \cdot \BFa
        \]
        Now take a $\BFx = \arg\max \mathcal{L}_{\BFzero, \BFx}$. Then, every $k$-step path ending at $\BFx$ gives
        \[
            k\delta \le \BFu \cdot \BFx \leq k\Delta.
        \]
        Moreover, since $\BFx \in \mathcal{R}_n$, there exists a path from the origin to $\BFx$ of length $n$, so $n\delta \le \BFu \cdot \BFx \le n\Delta$, which combined with previous inequality gives
        \[
            \frac{\delta n}{\Delta} \le k \le \frac{\Delta n}{\delta}.
        \]
        Consequently, 
        \[
            v_n(\BFzero) \le G_n \le \max_{k \in [\delta n/\Delta, \Delta n/\delta]} v_k(\BFzero).
        \]
        Invoking Theorem~\ref{thm:iid_and_oriented_percolation_games}, both outer terms converge to $v_{\infty}$ as $n \to \infty$ almost surely, proving the desired identity.   
    \end{proof}

    % \begin{remark}[Everything coincides for the canonical action set]
    %     In the quantity $\max_{\BFx \in \mathcal{R}_n} \mathcal{L}_{\BFzero,\BFx}$, the endpoint is restricted to points reachable in $n$ steps, but the number of steps used in the definition of $\mathcal{L}_{\BFzero,\BFx}$ is free. By contrast, in $v_{n}(\BFzero)$ the time horizon is fixed, but the endpoint is free. These two quantities coincide when every path from $\BFzero$ to a point $\BFx \in \mathcal{R}_n$ has exactly $n$ steps. In particular, for the canonical action set, every path to $\lfloor n\BFx \rfloor$ has length $\|\lfloor n\BFx \rfloor\|_1 \sim n\|\BFx\|_1$. Let $\mathcal{D} = \{\BFx \in \bbR_+^d:\|\BFx\|_1=1\}$. Using the fact that $ \mathcal{L}_{\BFzero, \lfloor n\BFx\rfloor} \to \mu(\BFx)/\|\BFx\|_{1}$ for $\BFx \bbNe 0$ as $n \to \infty$, concavity and homogeneity of $\mu$, Proposition~\ref{prop:p2p-uniform-value,prop:p2l-uniformvalue}, we get:
    %     \[
    %         \max_{\BFx \in \mathcal{D}}\ \lim_{n \to \infty} \mathcal{L}_{\BFzero, \lfloor n\BFx\rfloor} =  \max_{\BFx \in \mathcal{D}}\mu(\BFx) = \frac{\mu(\BFone_d)}{d} = v_{\infty} = \mu_{\mathrm{pl}} = \lim_{n \to \infty} \max_{\BFx \in \mathcal{R}_n} \mathcal{L}_{\BFzero, \BFx}, \ \Pbb\text{-a.s.}
    %     \]
    % \end{remark}
    
\section{Perspectives}\label{sec:perspectives}
    We first discuss a possible additional use of Condition \eqref{eq:casi-uniform} (Point 1), and an open question concerning model relaxations (Point 2). Then, future research directions concerning the connection with percolation theory are presented, along with a conjecture based on the analysis in Sections \ref{subsec:critical_probability_threshold} and \ref{subsec:simulations} (Point 3). The general structure of 0-optimal strategies in dimension 1 is discuted in Point 4. Finally, we raise two questions concerning the link with LPP (Point 5).

    \begin{enumerate}[noitemsep,topsep=0pt]
        \item Under Condition \eqref{eq:casi-uniform}, the proof of the Tauberian theorem in \citep{Ziliotto2016-tauberian} could potentially be adapted to obtain a version of the uniform Tauberian theorem from \citet{Lehrer1992} featuring uniform convergence on the set of states reachable within a polynomial distance in $n$, rather than on the whole $\bbZ^d$.
        \item A natural direction is to relax the i.i.d.\ assumption on the payoffs. Can convergence results be established when the payoffs are i.i.d.\ only across space? In potential applications, this assumption may be more realistic, since a decision maker may have correlated interests across different actions. Under orientability, the answer is yes by \citep[Theorem~2.1]{GarnierZiliotto2022}. The main difficulty therefore lies in controllable PMDPs. In this case, optimal strategies may no longer be cyclic. Instead, the player may need to follow self-avoiding paths, which introduces two challenges:~edge payoffs along such paths are non-stationary and the optimization must be performed over an uncountable set of paths.
        \item The analysis in Section \ref{sec:dimensional-lifting-oriented-percolation} suggests that PMDPs could offer a novel route for approximating critical probability parameters using sequences of low-dimensional PMDPs. Although established series-expansion methods by \citet{Jensen1999} yield high numerical precision, our formulation introduces an apparently simpler perspective on the problem. A comprehensive performance comparison and algorithmic refinement represent natural next steps for researchers interested in alternative parameter-approximation techniques.~Moreover, based on the analysis in Sections \ref{subsec:critical_probability_threshold} and \ref{subsec:simulations}, we propose the following conjecture.
        \begin{conjecture}
            For any oriented PMDP with $\mathrm{Bernoulli}(p)$ payoffs and $p \in [0, 1]$, 
            \[
                \text{$v_{\infty}(p) = 1 \quad$ if and only if $\quad p \ge p_c^*$},
            \]
            where $p_c^*$ is the critical parameter of an associated oriented percolation process defined on the same transition graph of the decision process.
        \end{conjecture}
        \item In all our examples of Section \ref{sec:dimension1} the limit value is a rational function by parts, and optimal strategies could be constructed using finite automatas (and the knowledge of $p$). Is this true for any Percolation MDP in dimension 1 and with Bernoulli payoffs?

        \item We believe that Percolation MDPs provide a natural setting in which to extend results from LPP (see Remark~\ref{rem:resultLPP}). The following questions are of particular interest: Is there a variational formulation for the uniform value of oriented PMDPs with positive linearly independent set of actions? Which PMDPs exhibit KPZ scaling?
    \end{enumerate}

	\appendix
    
    \section{Proof of Proposition \ref{propex13}}\label{sec:appendix}

	A difficulty is that, if the future begins with 0 or 10, one has to compare a strategy that starts by +3 with the same strategy replacing the first action by +1+1+1. These 2 strategies differ by a payoff of 1 and a length of 2, hence the better one depends on what happens at the very end of the stream of payoffs. As a consequence there is no hope of having a strategy that is optimal for every fixed horizon (even large enough). We solve this issue by considering instead $\lambda$-discounted games, with payoff
    \[
        \gamma_\lambda^{\omega}(z, \sigma) := \lambda \sum_{m=1}^{+\infty}(1-\lambda)^{m-1}g_{\omega}(z_m)    
    \]
    and value $v_\lambda^{\omega}(z)$. Recall that if $\omega$ is fixed and such that the uniform value exists, then by \citep[Lemma 5.4]{Renault2011}  $v_\lambda^{\omega}(z)$ tends to $v_{\infty}$ as well. Hence
    \begin{equation}\label{eq:limit_value_eq_discounted_value}
        v_{\infty} = \lim_{\lambda \to 0}v_{\lambda}(z), \; \Pbb\text{-a.s.}
    \end{equation}
    An additional benefit will be that computations of the discounted value will involve fixed points instead of recursive equations and hence will be easier. A small drawback is that, in contrast to finitely repeated games, the order of payoffs counts. For example, if the future begins with 10000111, the strategy $\sigma^*$ given in the Proposition recommends playing +3+3+1 but playing +1+3+3 is in fact marginally better (it gives a payoff of 1 then 0 then 1 instead of 0 then 1 then 1, which is strictly better since one discounts payoffs). Finding a strategy 0-optimal for every (small enough) $\lambda$ would thus be very cumbersome, and we will instead prove that $\sigma^*$ is $\lambda$-optimal in $\Gamma_\lambda(z)$ for every sufficiently small $\lambda$:
    \[
        \exists \lambda_0 \in (0,1]: \forall \lambda \leq \lambda_0, \; \gamma_\lambda(z, \sigma^*) \geq v_{\lambda}(z) - \lambda, \quad \Pbb\text{-a.s.}
    \]
    Letting $\lambda$ go to 0 will then imply that  $\sigma^*$ guarantees $v_\infty$ in $\Gamma(z)$. 
    
    Let us first give a sketch of the proof. It is easy to show that it is optimal to play $+3$ if the future begins with $00$, and to play $+1$ if the future begins with $11$. 
    A first question is, when the future begins with 01 or 10, whether the player prefers to start by playing $+3$ or +1+1+1. For a small $\lambda$ the difference between the first and second types of strategy will be of the order of $\lambda(-1+2v_\lambda)$, the $-1$ coming from the loss of a payoff of 1 during the first two steps, and the $2v_\lambda$ coming from the fact that by making an immediate jump of +3 instead of three small jumps one will have an earlier access to future payoffs. Hence, a type of strategy will be better than the other depending on whether $p$ is below or above the probability $p^*$ such that $v_\infty(p^*)=\frac{1}{2}$. 
    When $p<p^*$ one then verifies easily that the proposed myopic strategy is optimal. 
    Things are more complex when $p>p^*$: the preceding argument ensures that it is not optimal to start by +1+1+1, but it says nothing of strategies starting with +1+3 or +1+1+3. We then consider future payoff sequence beginning with either by $10^k1$ or $010^k1$. We first determine the optimal strategy for $k$ small. When $k$ is at least $3$, there are long sequences of $0$ in the future in which it is optimal to play $+3$, and this implies by induction that the optimal strategy depends only on $k$ modulo 3.
    
    We start with a simple technical lemma.
    
    \begin{definition}[almost optimality in $\lambda$-discounted games]
        For a fixed discount factor $\lambda$, a strategy $ \sigma$ is said to be almost optimal in $\Gamma_\lambda(z)$ if there almost surely exists an infinite sequence $0 = n_0 < n_1 < n_2 < \cdots$ such that the best deviation from $\sigma$ involves only swapping two payoffs in each of the blocks $[1,n_1],\ [n_1+1, n_2], \ldots$.
    \end{definition}
    
    \begin{lemma}\label{almostopt}
        If a strategy is almost optimal in $\Gamma_\lambda(z)$ for some discount factor $\lambda$, then it is $\lambda$-optimal in $\Gamma_\lambda(z)$. In particular, if a strategy is almost optimal in $\Gamma_\lambda(z)$ for all discount factors $\lambda$ small enough, then it is $0$-optimal in $\Gamma(z)$.
    \end{lemma}
    
    \begin{proof}
        Fix a  discount factor $\lambda$ and an almost optimal strategy $\sigma$ in $\Gamma_\lambda(z)$. 
        By definition, almost surely one can find $n_1<n_2<\cdots$ such that on each block $[n_k+1,n_{k+1}]$ the optimal deviation gains at most $\lambda((1-\lambda)^{n_k}-(1-\lambda)^{n_{k+1}-1})$ (if one exchanges a 0 in stage $n_k+1$ with a 1 in stage $n_{k+1}$). Hence, almost surely,
        \begin{equation*}
            v^\omega_\lambda(z) \le 
            \le \gamma^\omega_\lambda(z,\sigma)+\lambda \sum_{k=0}^{+\infty}((1-\lambda)^{n_k}-(1-\lambda)^{n_{k+1}}) =  \gamma^\omega_\lambda(z,\sigma)+\lambda.
        \end{equation*}  
    \end{proof}
    
    \begin{proof}~[of Proposition \ref{propex13}] The cases $p=0$ or $p=1$ are trivial so we assume from now on that $p\in(0,1)$.~We divide the proof in 6 steps.
    
    \underline{Step 1: comparison between +3 and +1+1+1}. Let $\lambda \in (0,1]$ be a discount factor and $z$ the starting state. Assume in this step that exactly one of $g_\omega(z+1)$ and $g_\omega(z+2)$ is equal to 1.
    Let $\sigma$ be any strategy starting with $+3$ and $\tilde{\sigma}$ the same strategy replacing the first $+3$ by $+1+1+1$. We compare the discounted payoff of the two strategies. By dynamic programming, 
    \[
        \gamma^\omega_\lambda(z,\sigma)=\lambda g_\omega(z+3)+(1-\lambda) \gamma^\omega_\lambda(z+3,\sigma),
    \]
    while the discounted payoff with $\tilde{\sigma}$ is
    \[
        \gamma^\omega_\lambda(z,\tilde{\sigma})=\lambda g_\omega(z+1)+\lambda(1-\lambda) g_\omega(z+2)+\lambda(1-\lambda)^2 g_\omega(z+3)+(1-\lambda)^3 \gamma^\omega_\lambda(z+3,\tilde{\sigma}).
    \]
    Since the two strategies differ only in the first moves, $ \gamma^\omega_\lambda(z+3,\sigma)=\gamma^\omega_\lambda(z+3,\tilde{\sigma})$. Hence, the difference between their two payoffs is 
    \begin{eqnarray*}
        \gamma^\omega_\lambda(z,\sigma)-\gamma^\omega_\lambda(z,\tilde{\sigma}) 
        % & = & -\lambda g_\omega(z+1)-\lambda(1-\lambda) g_\omega(z+2)+\lambda^2(2-\lambda)g_\omega(z+3)+(1-\lambda)\lambda(2-\lambda) \gamma^\omega_\lambda(z+3,\sigma)\\
                & = & \lambda\left[- 1+2  \gamma^\omega_\lambda(z+3,\sigma)+o(1)\right].
    \end{eqnarray*}
    Hence $\sigma$ cannot be optimal if $v_\infty(p)<\frac{1}{2}$, because then $2\gamma^\omega_\lambda(z+3,\sigma)$ would almost surely be less than 1 for $ \lambda$ small enough, and $\tilde{\sigma}$ would then be a profitable deviation. Similarly, $\tilde{\sigma}$ cannot be optimal if $v_\infty(p)>\frac{1}{2}$.
    
    By Proposition~\ref{prop:bernoulli-continuity} there exists $p_1, p_2$ with $0<p_1\leq p_2<1$ such that $v_\infty(p)<\frac{1}{2}$ on $[0, p_1)$, $v_\infty(p)=\frac{1}{2}$ on $[p_1,p_2]$ and $v_\infty(p)>\frac{1}{2}$ on $(p_2, 1]$. Because of the previous observation we will handle separately these cases.
    
    \underline{Step 2: optimal strategy when $0<p<p_1$.} By the previous step, for $\lambda$ small enough there is a profitable deviation from any strategy $\sigma$ such that $\sigma(10)=+3$ or $\sigma(01)=+3$ or, a fortiori, $\sigma(11)=+3$. On the other hand, if the future is 00 then +3 is at least as good as +1 for any $\lambda$: +1+1+1 is strictly worse than +3, +1+1+3 is weakly worse than +3+1+1 and +1+3 is weakly worse than +3+1. All of this implies that the 2-myopic  strategy $\sigma^*$ given in the proposition is optimal for $\lambda$ small enough and hence also in $\Gamma(z)$.
    
    \underline{Step 3: value when $0<p<p_1$.} We now compute $v_\infty(p)$ in that case. Denote $w_\lambda$ the expectation of the value of the $\lambda$-discounted game (which does not depend on the starting state $z$). By dynamic programming and since $\sigma^*$ is optimal for $\lambda$ small enough,
    \[
        w_\lambda=p(\lambda+(1-\lambda)w_\lambda)+p(1-p)(\lambda(1-\lambda)+(1-\lambda)^2 w_\lambda)+(1-p)^2(\lambda p+(1-\lambda)w_\lambda).
    \]
    which implies $w_\lambda=\frac{p+p(1-p)(1-\lambda)+p(1-p)^2}{1+p(1-p)(1-\lambda)}$ and thus that $v_\infty(p)=\frac{p(3-3p+p^{2})}{1+p-p^2}$.\\ Finally, by continuity $\frac{1}{2}=v_\infty(p_1)=\frac{p_1(3-3p_1+p_1^{2})}{1+p_1-p_1^2}$ which yields $p_1=p^*$.
    
    \underline{Step 4: optimal strategy when $p_2 < p < 1$}. This case is more complex since, for example if the future begins with $01$, Step 1 only implies that +1+1+1 is worse than +3, which does not settle whether +1 or +3 is better. Thus, we need to account for many different cases for the future.
    \begin{itemize}[noitemsep,topsep=0pt]
        \item If the future begins with 00, then exactly as in Step 2 it is optimal to play +3 for any $\lambda$.
        \item If it begins with 011, then +1+1+1 is strictly worse than +3 for $\lambda$ small enough by Step 1, +1+1+3 is strictly worse than +3+1+1 and +1+3 is strictly worse than +3+1. Hence +3 is optimal for $\lambda$ small enough.  
        \item If it begins with 0101, then +1+1+1 is worse than +3 by Step 1 for $\lambda$ small enough, +1+1+3 is the same as +3+1+1 and +1+3 is the same as +3+1. Hence +3 is optimal  for $\lambda$ small enough. 
        \item If the future begins with 01001, then +3+1 is the same as +1+3, and +3+3 is strictly worse than +1+1+3+1.  Hence +1 is optimal  for any $\lambda$.
        \item If it begins with $010^k1$ with $k\geq 3$, then any optimal strategy will play +3 at least once (if the player starts by +1+1 he then has to play +3) hence one can remove the three zeros at positions $z+3, z+4$ and $z+5$ and proceed by induction on $k$ to prove that $\sigma^*$ is optimal up to\footnote{Optimality is not guaranteed in that case since the three zeros we remove will correspond to a +3 sometimes played in the first step and sometimes later. Hence one could sometimes do marginally better by exchanging two payoffs (see the example before \ref{almostopt}).} maybe exchanging two payoffs on this block.
        \item If it begins with 11, then +1+1+1 is strictly better than +3 hence +1 is optimal for any $\lambda$.
        \item If it begins with 101, then +1+1+1 is strictly worse than +3 for $\lambda$ small enough by Step 1, +1+1+3 is weakly worse than +3+1+1, and +1+3 is the same as +3+1. Hence +3 is optimal  for $\lambda$ small enough. 
        \item If it begins with 1001, then +3+1 is strictly worse than +1+3, and +3+3 is strictly worse than +1+3+1+1. Hence +1 is optimal for any $\lambda$.
        \item Finally, if the future begins with $10^k1$ with $k\geq3$, then any optimal strategy will play +3 at least once (if the player starts by +1 he then has to play +3)  hence one can remove the three zeros at positions $z+2, z+3$ and $z+4$ and proceed by induction on $k$  to prove that $\sigma^*$ is optimal up to\footnote{See previous footnote.} maybe exchanging two payoffs on this block.
    \end{itemize}
    This proves that the $\sigma^*$ from the proposition is almost optimal for $\lambda$ small enough and hence by \ref{almostopt} is optimal in $\Gamma(z)$.
    
    \underline{Step 5:  value when $p_2<p<1$}. Let us now compute the value $v_\infty(p)$ in that case. Let $w_\lambda$ denote the expected payoff under $\sigma^*$ in the $\lambda$-discounted game; this expectation does not depend on the starting state $z$. Let $w_{\lambda}^1$ and $w_{\lambda}^0$ denote the corresponding conditional expected payoffs  given that the future begins with 1 and 0, respectively. We first derive a recursive formula for $w_{\lambda}^1$ using dynamic programming. Given that the future begins with $1$, then for any $k\geq 0$ the future begins with $10^k1$ with probability $p(1-p)^k$ where for convenience we denote $q=1-p$. Then:
    \begin{itemize}[noitemsep,topsep=0pt]
        \item If $k=3k'$, $\sigma^*$ plays +1 (with a payoff of 1), then $k'$ times +3 (with $k'$ payoffs of 0). Then we are at a position with future 1, and hence expected payoff $w_{\lambda}^1$.
        \item If $k=3k'+1$, $\sigma^*$ plays $k'$ times +3 (with $k'$ payoffs of 0). Then once more +3 (with a payoff of 1), and then we are at a position with no conditioning on the future payoffs, and hence expected payoff $w_\lambda$.
        \item If $k=3k'+2$, $\sigma^*$ plays +1 (with a payoff of 1), then $k'$ times +3 (with $k'$ payoffs of 0). Then once more +3 (with a payoff of 1), and then we are at a position with no conditioning on the future payoffs, and hence expected payoff $w_\lambda$.
    \end{itemize}
    Thus, we get:
    \begin{eqnarray*}
        w_{\lambda}^1 & = & \sum_{k'=0}^{+\infty} pq^{3k'}(\lambda+(1-\lambda)^{k'+1}w_{\lambda}^1) + \sum_{k'=0}^{+\infty}pq^{3k'+1}(\lambda(1-\lambda)^{k'}+(1-\lambda)^{k'+1}w_\lambda)\\
                & + & \sum_{k'=0}^{+\infty}pq^{3k'+2}(\lambda+\lambda(1-\lambda)^{k'+1}+(1-\lambda)^{k'+2}w_\lambda).
    \end{eqnarray*}
    Denote $a_\lambda:=w_{\lambda}^1-w_\lambda$. Regrouping terms,
    \begin{eqnarray*}
        w_\lambda+a_\lambda & = & p a_\lambda \sum_{k'=0}^{+\infty} q^{3k'}(1-\lambda)^{k'+1} + p w_\lambda  \sum_{k'=0}^{+\infty} (1-\lambda)^{k'+1} [q^{3k'}+q^{3k'+1}+(1-\lambda)q^{3k'+2}]\\
        & + & \lambda p \sum_{k'=0}^{+\infty}q^{3k'}+q^{3k'+1}(1-\lambda)^{k'}+q^{3k'+2}(1+(1-\lambda)^{k'+1})\\
        & = & p a_\lambda\left(\frac{1-\lambda}{1-(1-\lambda)q^3}\right) + p w_\lambda \left(\frac{(1-\lambda)+(1-\lambda)q+(1-\lambda)^2q^2}{1-(1-\lambda)q^3}\right)\\
        & + & \lambda p \left(\frac{1+q^2}{1-q^3}+\frac{q+(1-\lambda)q^2}{1-(1-\lambda)q^3}\right).
    \end{eqnarray*}
    We will now neglect all terms  in $o(\lambda)$, using that $|w_\lambda|\leq 1$, that (by a standard coupling argument) $|a_\lambda|\leq \lambda$:
    \begin{equation*}
        w_\lambda+a_\lambda = a_\lambda\left(\frac{p}{1-q^3}\right) + w_\lambda \left(1-\lambda\frac{p(1+q+2q^2)+q^3}{1-q^3}\right) +\lambda p \left(\frac{1+q+2q^2}{1-q^3}\right)+o(\lambda),
    \end{equation*}
        which becomes, after subtraction of $w_\lambda$, division by $\lambda$ and multiplication by $1-q^3$:
    \begin{equation}\label{eqfuture1}
        w_\lambda(1+q^2-q^3)+\frac{a_\lambda}{\lambda}q(1+q)(1-q)= 1+q^2-2q^3+o(1).
    \end{equation}
    
    Let us similarly give a recursive formula for $w_{\lambda}^0$ (conditional expected payoff using $\sigma^*$ given that the future starts with $0$), using dynamic programming. Given that the future begins with $0$, then the future begins with $00$ with probability $q$ and for any $k\geq 0$ the future begins with $010^k1$ with probability $p^2q^k$. Hence, one obtains by definition of $\sigma^*$:
    \begin{eqnarray*}
        w_{\lambda}^0 = q(\lambda p+(1-\lambda)w_\lambda) & + & \sum_{k'=0}^{+\infty} p^2q^{3k'}(\lambda(1-\lambda)^{k'}+(1-\lambda)^{k'+1}w_\lambda) + \sum_{k'=0}^{+\infty} p^2q^{3k'+1}((1-\lambda)^{k'+1}w_{\lambda}^0)\\
        & + &\sum_{k'=0}^{+\infty} p^2q^{3k'+2}(\lambda(1-\lambda)+\lambda(1-\lambda)^{k'+2}+(1-\lambda)^{k'+3}w_\lambda).
    \end{eqnarray*}
    Denote $b_\lambda:=w_\lambda-w_{\lambda}^0$. Regrouping terms,
    \[
        \begin{aligned}
            w_\lambda-b_\lambda ={}&\frac{a_\lambda qp^2(1-\lambda)}
                    {1-(1-\lambda)q^3}\\
                &+w_\lambda\left((1-\lambda)q+\frac{p^2\bigl((1-\lambda (1+q)+(1-\lambda)^3q^2\bigr)}{1-(1-\lambda)q^3}
            \right)\\
                &+\lambda pq+\lambda p^2\left(
            \frac{(1-\lambda)q^2}{1-q^3}
            +\frac{1+(1-\lambda)^2q^2}{1-(1-\lambda)q^3}
            \right).
        \end{aligned}
    \]
    We once again neglect all terms  in $o(\lambda)$, using in addition that (by a standard coupling argument) $|b_\lambda|\leq \lambda$:
    \[
        \begin{aligned}
            w_\lambda-b_\lambda = & \ a_\lambda\frac{qp}{1+q+q^2} + w_\lambda \left(1-\lambda q-\lambda p \frac{1+q+3q^2}{1+q+q^2}-\lambda \frac{q^3}{1+q+q^2} \right) \\
            & \quad + \lambda p \left(q+\frac{1+2q^2}{1+q+q^2}\right)+o(\lambda),
        \end{aligned}
    \]
    which becomes, after subtraction of $w_\lambda$, division by $\lambda$ and multiplication by $1+q+q^2$:
    \begin{equation}\label{eqfuture0}
        q(1-q) \frac{a_\lambda}{\lambda}+(1+q+q^2)\frac{b_\lambda}{\lambda}+1+2q^2-2q^3-q^4=w_\lambda(1+q+3q^2-q^3)+o(1).
    \end{equation}
    
    Finally, we get a third equation by simply conditioning on the immediate future:
    \[
        w_\lambda=pw_{\lambda}^1+qw_{\lambda}^0,
    \]
    hence $pa_\lambda=q b_\lambda$. Replacing $a_\lambda$ in \ref{eqfuture1} and \ref{eqfuture0} one gets the system
    \begin{eqnarray*}
        w_\lambda(1+q^2-q^3)+\frac{b_\lambda}{\lambda}q^2(1+q) & = & 1+q^2-2q^3+o(1)\\
        (1+q+2q^2)\frac{b_\lambda}{\lambda}+1+2q^2-2q^3-q^4 & = & w_\lambda(1+q+3q^2-q^3)+o(1),
    \end{eqnarray*}
    which yields 
    \begin{equation*}
        w_\lambda = \frac{p\left(28 - 90p + 127p^2 - 98p^3 + 43p^4 - 10p^5 + p^6\right)}{12 - 29p + 25p^2 - 2p^3 - 10p^4 + 6p^5 - p^6}+o(1),
    \end{equation*}
    and the result follows by letting $\lambda$ tend to 0. Finally, by continuity 
    \[
        \frac{1}{2}=v_\infty(p_2)=\frac{p_2\left(28 - 90p_2 + 127p_2^2 - 98p_2^3 + 43p_2^4 - 10p_2^5 + p_2^6\right)}{12 - 29p_2 + 25p_2^2 - 2p_2^3 - 10p_2^4 + 6p_2^5 - p_2^6}.
    \] 
    Hence $p_2$ is a root of $2X^7-19X^6+80X^5-186X^4+256X^3-205X^2+85X-12$ which factorizes into $(2X^3-5X^2+5X-1)(X^4-7X^3+20X^2-25X+12)$ and has only $p^*$ as a real root. Hence $p_2=p^*$.
    
    \underline{Step 6: no optimal myopic strategy when $p^*<p<1$}. A full proof would involve cumbersome computations so we only give a sketch of the proof. Fix $k\geq2$ and let $\sigma$ be a $k$-myopic optimal strategy. First, note that $\sigma$ must satisfy $\sigma(11)=+1$, or else +1+1+1 would be a strictly profitable deviation from $+3$ in some block of length $k$ starting with $11$. A consequence is that whenever there is a sequence of 4 consecutive states with a payoff of 1, any optimal play must at some point land on the third of these four states.
    
    Consider now $\sigma(10^{k-2})$. There are two possible cases.
    
    Assume first that $\sigma(10^{k-2})=+3$. Fix now  $k'$ such that $3k'\geq k-2$ and consider any state $z$ such that future payoffs start by $10^{3k'+2}1001111$. By the previous observation, the play will at some point land on the penultimate 1, which is state $z+3k'+9$.  Let $m$ be the number of $+1$ moves played between state $z$ and state $z+3k'+7$. Clearly $m=0\Mod 3$, and it is easy to see that neither $m=0$ nor $m\geq 6$ are optimal, hence $m=3$. Because of the starting move of $+3$, the best outcome $\sigma$ can achieve is then to collect 4 payoffs of 1 and $k'+1$ payoffs of 0 between state $z$ and state $z+3k'+9$. However a play of $+1$, then $k'+2$ consecutive $+3$ and then two $+1$ would collect $5$ payoffs of 1 and  $k'$ payoffs of 0 between state $z$ and state $z+3k'+8$, which is strictly better. Since $p<1$, a block consisting of $10^{3k'+2}1001111$ will happen with positive frequency, and we can thus construct a strictly profitable deviation from $\sigma$ by playing instead $k'+2$ consecutively $+3$ and then $2$ moves of $+1$ during every such block.
    
    Assume now that $\sigma(10^{k-2})=+1$. Fix $k'$ such that $3k'+1\geq k-2$ and consider any state $z$ such that future payoffs starts by $10^{3k'+1}1001111$. Once again, the play will at some point land on the penultimate 1, which is state $z+3k'+8$. Let $m$ be the number of $+1$ moves played between state $z$ and state $z+3k'+8$. Clearly $m=2\Mod 3$. If $m=2$, because of the starting move of $+1$, the best outcome $\sigma$ can achieve is to collect 3 payoffs of 1 and $k'+1$ payoffs of 0 between state $z$ and state $z+3k'+8$. If $m=5$, the best outcome $\sigma$ can achieve is to collect 5 payoffs of 1 and $k'+1$ payoffs of 0 between state $z$ and state $z+3k'+8$, and the situation is strictly worse if $m\geq 8$. However, a play of $k'+2$ consecutively $+3$ and then $2$ moves of $+1$ would collect $4$ payoffs of 1 and  $k'$ payoffs of 0 between state $z$ and state $z+3k'+8$. This is clearly strictly better than 3 payoffs of 1 and $k'+1$ payoffs of 0 ; and because we assumed $p>p^*$ this is also strictly better than 5 payoffs of 1 and $k'+1$ payoffs of 0. Since $p<1$, a block consisting of $10^{3k'+1}1001111$ will happen with positive frequency, and we can thus construct a strictly profitable deviation from $\sigma$ by playing instead $k'+2$ consecutively $+3$ and then $2$ moves of $+1$ during every such block. 
    \end{proof}

    \vspace{0.5cm}
    \textbf{Acknowledgments.} The authors acknowledge the support of the ZADYG Project (ANR-25-CE48-7058) of the French Agence Nationale de la Recherche (ANR), and warmly thank Sylvain Sorin and Bruno Ziliotto for their helpful suggestions.~The first author also thanks Avelio Sepúlveda and Jérôme Renault for their advice and insightful questions; and acknowledges the support of the Cofund MathPhDInFrance Program, with $\textit{disclaimer:}$~Co-funded by the European Union.~Views and opinions expressed are, however, those of the author only and do not necessarily reflect those of the European Union.~Neither the European Union nor the granting authority can be held responsible for them.

    \bibliography{references}

\end{document}

%% file: preamble.tex
\usepackage[margin=1in]{geometry}

\usepackage{setspace}
\usepackage{xcolor}
\usepackage{graphicx}
\usepackage{subcaption}
\usepackage{appendix}
\usepackage{enumitem}
\usepackage{todonotes}

\usepackage{amsmath,amssymb,amsthm}
\usepackage{mathrsfs} % para \mathscr
\usepackage{mathtools} % para usar \floor..
\usepackage{bbm}
\usepackage{dsfont}

\usepackage{tikz}
\usetikzlibrary{arrows.meta,decorations.pathmorphing}

\usepackage{tocloft}
\usepackage{titlesec}
\titleformat{\section}
  {\normalfont\bfseries}
  {\thesection}
  {1em}
  {}
\titleformat{\subsection}
  {\normalfont\bfseries}
  {\thesubsection}
  {1em}
  {}

\usepackage{fancyhdr}
\fancypagestyle{plain}{
    \fancyhf{}
    \fancyfoot[C]{\thepage}

}

\renewenvironment{proof}
    {\textit{Proof.}\setlength{\parindent}{0pt}}
    {\qed}

\newcommand{\BFz}{\mathbf{z}}
\newcommand{\BFzero}{\mathbf{0}}
\newcommand{\BFe}{\mathbf{e}}
\newcommand{\BFone}{\mathbf{1}}
\newcommand{\BFa}{\mathbf{a}}
\newcommand{\BFx}{\mathbf{x}}
\newcommand{\BFy}{\mathbf{y}}
\newcommand{\BFu}{\mathbf{u}}

\newcommand{\BFq}{\mathbf{q}}

\newcommand{\bbR}{\mathbb{R}}
\newcommand{\bbQ}{\mathbb{Q}}
\newcommand{\bbZ}{\mathbb{Z}}
\newcommand{\bbN}{\mathbb{N}}
\newcommand{\E}{\mathbb{E}}
\newcommand{\Pbb}{\mathbb{P}}
\newcommand{\F}{\mathscr{F}}

\newcommand{\Mod}[1]{\ (\mathrm{mod}\ #1)}

\makeatletter
\providecommand\@dotsep{5}
\renewcommand{\listoftodos}[1][\@todonotes@todolistname]{%
    \@starttoc{tdo}{#1}
}
\makeatother

\usepackage{hyperref}

\hypersetup{
    linkbordercolor=red,      % Azul para Índice, Secciones, Ecuaciones y Teoremas
    citebordercolor=green,       % Rojo para los números [1] de las citas
    urlbordercolor=white,       % Turquesa para URLs externos
    linktoc=all          % Obliga a todo el renglón del índice a ser cliqueable
}

\theoremstyle{plain}
\newtheorem{theorem}{Theorem}[section]
\newtheorem{lemma}[theorem]{Lemma}

\newtheorem{proposition}[theorem]{Proposition}
\newtheorem{conjecture}[theorem]{Conjecture}

\theoremstyle{definition}
\newtheorem{definition}[theorem]{Definition}
\newtheorem{remark}[theorem]{Remark}
\newtheorem{example}{Example}

\usepackage[sort&compress]{natbib}
\bibpunct[, ]{[}{]}{,}{n}{}{,}%